\documentclass{aims} 
\usepackage{paralist}
\usepackage{subfig}
\usepackage{algorithm}
\usepackage{algorithmic}
\usepackage{multirow}
\usepackage{subfig}
\usepackage{geometry}
\usepackage{setspace}
\usepackage{esint}
\usepackage{extarrows}
\usepackage{mathrsfs}
\usepackage{algorithm}
\usepackage{algorithmic}
\usepackage{booktabs}
\usepackage{graphicx}
\usepackage{amsmath,amssymb,amsthm,enumerate,mathrsfs}
\usepackage[misc]{ifsym}
\usepackage{epsfig} 
\usepackage{epstopdf} 
\usepackage[colorlinks=true]{hyperref}
\hypersetup{urlcolor=blue, citecolor=red}
\allowdisplaybreaks

\def\currentvolume{X}
 \def\currentissue{X}
  \def\currentyear{200X}
   \def\currentmonth{XX}
    \def\ppages{X-XX}
     \def\DOI{10.3934/xx.xxxxxxx}

\newtheorem{theorem}{Theorem}[section]

\newtheorem{lemma}[theorem]{Lemma}
\newtheorem{proposition}[theorem]{Proposition}

\theoremstyle{definition}
\newtheorem{definition}[theorem]{Definition}
\newtheorem{remark}[theorem]{Remark}
\newtheorem{assumption}[theorem]{Assumption}

\title[Running head]
{Smoothing Accelerated Proximal Gradient Method with Fast Convergence Rate  for Nonsmooth Multi-objective Optimization} 

\author[First-name1 Last-name1 and First-name2 Last-name2]{Huang Chengzhi}


\thanks{The first author is supported by [insert grant information here]}

\thanks{$^*$Corresponding author: First-name1 Last-name1}

\begin{document}
\maketitle

\centerline{\scshape
Huang Chengzhi$^{{\href{mailto:2022110518015@cqnu.com}{\textrm{\Letter}}}*1}$
and First-name2 Last-name2$^{{\href{mailto:author2@domain.com}{\textrm{\Letter}}}1,2}$}

\medskip

{\footnotesize
 \centerline{$^1$Chong qing normal university, China}
} 

\medskip

{\footnotesize
 \centerline{$^2$Affiliation, Country}
}

\bigskip

 \centerline{(Communicated by Handling Editor)}


\begin{abstract}
This paper proposes a Smoothing Accelerated Proximal Gradient Method with Extrapolation Term (SAPGM) for nonsmooth multiobjective optimization. By combining the smoothing methods and the accelerated algorithm for multiobjective optimization by Tanabe et al., our method achieve fast convergence rate. Specifically, we establish that the convergence rate of our proposed method can be enhanced to \(o(\ln^\sigma k/k)\) by incorporating a  extrapolation term \(\frac{k-1}{k + \alpha -1}\) with \(\alpha > 3\).Moreover, we prove that the iterates sequence is convergent to a Pareto optimal solution of the primal problem. Furthermore, we present an effective strategy for solving the subproblem through its dual representation, validating the efficacy of the proposed method through a series of numerical experiments.
\end{abstract}


\section{Introduction}
Multiobjective optimization involves the simultaneous minimization (or maximization) of multiple objective functions while considering relevant constraints. The concept of Pareto optimality becomes crucial, as finding a single point that minimizes all objective functions concurrently is challenging. A point is deemed Pareto optimal if there exists no other point with the same or smaller objective function values and at least one strictly smaller objective function value. Applications of multiobjective optimization are pervasive, spanning economics\cite{chinchuluun2007survey}, engineering\cite{hakanen2021multiobjective}, mechanics\cite{ren2021effects}, statistics\cite{rosinger1981interactive}, internet routing\cite{doolittle2018robust}, and location problems\cite{belgasmi2008evolutionary}.

This paper focuses predominantly on composite nonsmooth multiobjective optimization, expressed as:
\begin{equation}\label{1}
	\min_{x \in \mathbb{R}^{n}} F(x)
\end{equation}
with $F : \mathbb{R}^n \to (\mathbb{R} \cup \{ \infty\})^{m}$ and $F := (F_1,\dotsb,F_{m})^{T}$ taking the form
\begin{equation}\label{2}
	F_{i}(x) := f_{i}(x) + g_{i}(x), i = 1,2,\dotsb,m,
\end{equation}
where, $f_{i}:\mathbb{R}^{n} \to \mathbb{R}$ represents a convex but nonsmooth function, and
$g_{i}:\mathbb{R}^{n} \to \mathbb{R}$ is a closed, proper, and convex function,which may not be nonsmooth.

The composite optimization problem is a significant class of optimization problems, not only because it encompasses various practical challenges—such as minimax problems \cite{xian2021minimax} and penalty methods for constrained optimization \cite{xia2020penalty}—but also due to its wide range of applications. For instance, as discussed in \cite{tanabe2019proximal}, the separable structure in (\ref{1}) can be used to model robust multi-objective optimization problems, which involve uncertain parameters and optimize for the worst-case scenario. Additionally, this structure is applicable in machine learning\cite{jin2006multi}, particularly for solving multi-objective clustering problems.

Naturally, we are interested in methods for solving multi-objective optimization problems.
Common methods include scalarization method, evolutionary method and gradient method.

Scalarization is a fundamental approach to solve multiobjective optimization problems, transforming them into single-objective ones. Various procedures, such as optimizing one objective while treating others as constraints\cite{nikulin2012new}, or aggregating all objectives\cite{rocca2022sensitivity}, are commonly applied. Evolution algorithms\cite{wang2023dynamic} provide another avenue, but proving their convergence rate poses challenges. Consequently, traditional methods for solving the problem directly are also employed.

In response to limitations, descent methods for multiobjective optimization problems have gained significant attention. These algorithms, which reduce all objective functions at each iteration, offer advantages such as not requiring prior parameter selection and providing convergence guarantees under reasonable assumptions. Noteworthy methods include the steepest descent\cite{fliege2000steepest}, projected gradient\cite{fukuda2013inexact}, proximal point\cite{bonnel2005proximal}, Newton\cite{fliege2009newton}, trust region\cite{carrizo2016trust}, and conjugate gradient methods\cite{lucambio2018nonlinear} for solving
$g(x) = 0$. Among these, first-order methods, utilizing only the first-order derivatives of the objective functions, are distinguished, such as the steepest descent, projected gradient, and proximal gradient methods. The latter method converges to Pareto solutions with a rate of O(1/k).

To enhance the convergence efficiency of the proximal gradient method, numerous scholars have endeavored to introduce acceleration techniques into single-objective first-order methodologies. Detailed works can be seen in the following literature:\cite{nesterov1983unconstrained}\cite{beck2009fast}\cite{chambolle2015convergence}\cite{attouch2016rate}.

The application of acceleration algorithms in single objective scenarios prompted a significant surge in interest in exploring their efficacy in the realm of multi-objective optimization problems. A recent noteworthy development by Tanabe et al. \cite{tanabe2023accelerated} involves the extension of the highly regarded Fast Iterative Shrinkage-Thresholding Algorithm (FISTA) to the multi-objective context. The ensuing convergence rate, denoted as \(O(1/k^2)\) and characterized by a merit function \cite{tanabe2023new}, represents a substantial improvement over the proximal gradient method for Multi-Objective Problems (MOP) \cite{tanabe2019proximal}. Moreover, Nishimura et al. \cite{nishimura2022monotonicity} have established a monotonicity version of the multiobjective FISTA, adding to the methodological advancements in this domain. Furthermore, Tanabe et al. \cite{tanabe2022globally} have expanded the applicability of the multiobjective FISTA by introducing hyperparameters, offering a generalization applicable even in single-objective scenarios. Importantly, this extended framework preserves the commendable convergence rate of \(O(1/k^2)\) observed in the multiobjective FISTA. Additionally, it is proved that the iterative sequences is convergent.Inspired by the impact of the extrapolation parameters in single-objective case\cite{attouch2016rate}, we introduce the extrapolation parameter $\frac{k-1}{k+\alpha-1}$ with $\alpha > 3$  into the multiobjective proximal gradient algorithm.

After solving the problem of algorithm acceleration, another problem follows: how to deal with non-smooth multi-objective optimization efficiently?

For non-smooth multi-objective optimization problems, current research mainly includes Mäkelä et al.'s proximal bundle method \cite{makela2003multiobjective,montonen2018multiple,haarala2004newlimited,makela1992nonsmooth} and the subgradient method. Gebken et al. \cite{gebken2021efficient} proposed a subgradient descent algorithm for solving non-smooth multi-objective optimization problems by combining the descent direction from \cite{mahdavi2012effective} with the approximation based on Goldstein's $\epsilon$-subdifferential from \cite{goldstein1977optimization}. Besides, Konstantin Sonntag et al. \cite{kumari2015subgradient} proposed a new subgradient method for solving non-smooth vector optimization problems, which includes regularization and interior point variants of Newton's method. However, all of them face the challenge of requiring complex calculations and numerous calls to subgradients, resulting in a significant increase in computation time. Fortunately, Chen \cite{chen2012smoothing} proposed a smoothing construction, which used a sequence of functions to approximate the objective functions of the primal problem. This construction can avoid calculating the subgradient and directly use the gradient of the smoothing function to obtain the result. Inspired by this idea, we decided to construct an algorithm with fast speed to solve non-smooth multi-objective optimization problems under the smoothing framework, combined with the previously mentioned accelerated proximal gradient method with extrinsic terms.

Moreover, with practical computational efficiency in mind, we derive a convex and differentiable dual of the subproblem, simplifying its solution, particularly when the number of objective functions is fewer than the decision variable dimension. The entire algorithm is implemented using this dual problem, and its effectiveness is confirmed through numerical experiments.

The structure of this paper unfolds as follows: Section 2 introduces notations and concepts, Section 3 presents the smoothing accelerated proximal gradient method with extrapolation for nonsmooth multiobjective optimization, and Section 4 analyzes its
$o(\ln^\sigma k/k)$ convergence rate. Section 5 outlines an efficient method to solve the subproblem through its dual form, and Section 6 reports numerical results for test problems.

\section{Preliminary Results}

In this paper, for any natural number \(n\), the symbol \(\mathbb{R}^{n}\) denotes the \(n\)-dimensional Euclidean space. The notation \(\mathbb{R}^{n}_{+} \subseteq \mathbb{R}^{n}\) is employed to signify the non-negative orthant of \(\mathbb{R}^{n}\), denoted as \(\mathbb{R}^{n}_{+} := \{ v \in \mathbb{R}^{n} | v_{i} \geq 0, i=1,2,\dotsb,n \}\). Additionally, \(\Delta^{n}\) represents the standard simplex in \(\mathbb{R}^{n}\) and is defined as
\[
\Delta^{n} := \{ \lambda \in \mathbb{R}^{n}_{+} | \lambda_{i} \geq 0, \sum_{i=1}^{n} \lambda_{i} = 1 \}.
\]
Subsequently, the partial orders induced by \(\mathbb{R}^{n}_{+}\) are considered, where for any \(v^{1}, v^{2} \in \mathbb{R}^{n}\), \(v^{1} \leq v^{2}\) (alternatively, \(v^{1} \geq v^{2}\)) holds if \(v^{2} - v^{1} \in \mathbb{R}^{n}_{+}\), and \(v^{1} < v^{2}\) (alternatively, \(v^{1} > v^{2}\)) if \(v^{2} - v^{1} \in \text{int} \, \mathbb{R}^{n}_{+}\). Moreover, let \(\left\langle \cdot, \cdot \right\rangle\) denote the Euclidean inner product in \(\mathbb{R}^{n}\), specifically defined as \(\left\langle u, v \right\rangle := \sum_{i=1}^{n} u_{i}v_{i}\). The Euclidean norm \(\left\| \cdot \right\|\) is introduced as \(\left\| u \right\| := \sqrt{\left\langle u, u \right\rangle}\). Furthermore, the \(\ell_1\)-norm and the \(\ell_{\infty}\)-norm are defined by \(\left\| u \right\|_1 := \sum_{i=1}^{n} |u_{i}|\) and \(\left\| u \right\|_{\infty} := \max_{i=1,\dotsb,n} |u_{i}|\), respectively.

Because the construction of proximal gradient algorithm,we should introduce some basic definitions for following discussion.For a closed , proper and convex function $h:\mathbb{R}^{n} \to \mathbb{R} \cup \{ \infty \}$,the Moreau envelope of $h$ defined by
$$\mathcal{M}_{h}(x) := \min_{y \in R^{n}} \{h(y) + \frac{1}{2} \left\|x-y\right\|^2\}.$$
The unique solution of the above problem is called the proximal operator of $h$ and write it as
$$ prox_{h}(x) := \arg\min_{y \in R^{n}}\{h(y) + \frac{1}{2} \left\|x-y\right\|^2\}.$$

Next,we introduce a property between Moreau envelope and proximal operation by following lemma.
\begin{lemma}[\cite{rockafellar1997convex}]
	If  {$h$} is a proper closed and convex function, the Moreau envelope
	$\mathcal{M}_{h}$ is lipschitz continuous and takes the following form,
	$$\nabla \mathcal{M}_{h}(x) := x- prox_{h}(x).$$
\end{lemma}

As explicated in the Introduction section, the principal challenge in addressing the optimization problem denoted as (\ref{1}) through the Proximal Gradient (PG) and Accelerated Proximal Gradient (APG) methods comes from the nonsmooth nature of the objective function \(f\). Specifically, when \(f\) is nonsmooth or its gradient \(\nabla f\) lacks global Lipschitz continuity, a straightforward approach involves resorting to the smoothing method, a pivotal aspect in our analytical framework. In the context of this study, we introduce an algorithm utilizing the smoothing function delineated in \cite{chen2012smoothing}. This smoothing function serves the purpose of approximating the nonsmooth convex function \(f\) by a set of smooth convex functions, thereby facilitating the application of gradient-based optimization techniques.
\begin{definition}[\cite{chen2012smoothing}]\label{def1}
	For convex function $f$ in (\ref{2}), we call $\tilde{f}:\mathbb{R}^{n} \times \mathbb{R}_{+} \to \mathbb{R}$ a smoothing
	function of $f$, if $\tilde{f}$ satisfies the following conditions:
	
	(i) for any fixed $\mu > 0$,$\tilde{f}( \cdot, \mu)$ is continuously differentiable on $\mathbb{R}^{n}$;
	
	(ii) $\lim_{z \to x,\mu \downarrow 0}\tilde{f}(z,\mu) = f(x), \forall x \in \mathbb{R}^{n}$;
	
	(iii) (gradient consistence) $\{\lim_{z \to x,\mu \downarrow 0}\tilde{c}(z,\mu)\} \subseteq \partial f(x), \forall x \in \mathbb{R}^{n}$ ;
	
	(iv) for any fixed $\mu >0$, $\tilde{f}(z,\mu)$ is convex on $\mathbb{R}^{n}$;
	
	(v) there exists a $k > 0$ such that
	$$|\tilde{f}(x,\mu_2) - \tilde{f}(x,\mu_1)| \leq k|\mu_1-\mu_2|, \forall x \in \mathbb{R}^{n}, \mu_1,\mu_2 \in \mathbb{R}_{++};$$
	
	(vi) there exists an $L > 0$ such that $\nabla_{x}\tilde{f}(\cdot,\mu)$ is Lipschitz continuous on $\mathbb{R}^{n}$ with
	factor $L\mu^{-1}$ for any fixed $\mu \in \mathbb{R}_{++}$.
\end{definition}

Combining properties (ii) and (v) in Definition (\ref{def1}), we have

$$|\tilde{f}(x,\mu) - f(x)| \leq k \mu,\forall x \in \mathbb{R}^{n},\mu \in \mathbb{R}_{++}.$$

The exploration of smooth approximations for diverse specialized nonsmooth functions has a venerable lineage, yielding a wealth of theoretical insights \cite{chen2012smoothing}, \cite{facchinei2003finite}, \cite{nesterov2005smooth}, \cite{rockafellar1998variational}, \cite{hiriart1996convex}. The foundational conditions (i)–(iii) articulated herein are integral elements in the characterization of a smoothing function, as delineated in \cite{chen2012smoothing}. These conditions are imperative for ensuring the efficacy of smoothing methods when applied to the resolution of corresponding nonsmooth problems. Condition (iv) stipulates that the smoothing function \(\tilde{f}(\cdot, \mu)\) preserves the convexity of \(f\) for any fixed \(\mu \in \mathbb{R}_{++}\). Conditions (v) and (vi) serve to guarantee the global Lipschitz continuity of \(\tilde{f}(x, \cdot)\) for any fixed \(x \in \mathbb{R}^{n}\) and the global Lipschitz continuity of \(\nabla_{x}\tilde{f}(\cdot, \mu)\) for any fixed \(\mu \in \mathbb{R}_{++}\), respectively. These conditions collectively establish a foundation for the utility and effectiveness of the smoothing function in the context of nonsmooth optimization problems.

We now revisit the optimality criteria for the multiobjective optimization problem denoted as (\ref{1}). An element \(x^{*} \in \mathbb{R}^{n}\) is deemed weakly Pareto optimal if there does not exist \(x \in \mathbb{R}^{n}\) such that \(F(x) < F(x^{*})\), where \(F: \mathbb{R}^{n} \to \mathbb{R}^{m}\) represents the vector-valued objective function. The ensemble of weakly Pareto optimal solutions is denoted as \(X^{*}\). The merit function \(u_{0}: \mathbb{R}^{n} \to \mathbb{R} \cup \{\infty\}\), as introduced in \cite{tanabe2023new}, is expressed in the following manner:
\begin{equation}\label{3}
	u_0(x) := \sup_{z \in \mathbb{R}^{n}} \min_{i =1,\dotsb,m}[F_{i}(x) - F_{i}(z)].
\end{equation}
The following lemma proves that $u_0$ is a merit function in the Pareto sense.
\begin{lemma}[\cite{tanabe2023accelerated}]
	Let $u_0$ be given as (\ref{3}), then $u_0(x) \geq 0, x \in \mathbb{R}^{n}$, and $x$ is the
	weakly Pareto optimal for (\ref{1}) if and only if $u_0(x) = 0$.
\end{lemma}

\section{The Smoothing Accelerated Proximal Gradient Method with Extrapolation term for Non-smooth Multi-objective Optimization}
This section introduces an accelerated variant of the proximal gradient method tailored for multiobjective optimization. Drawing inspiration from the achievements reported in \cite{attouch2016rate}, we incorporate extrapolation techniques with parameters \(\beta_{k} = \frac{k-1}{k + \alpha - 1}\), where \(\alpha > 3\). Choosing the smoothing function \(\tilde{c}\) as defined in Definition (\ref{def1}), we formulate an accelerated proximal gradient algorithm to solve the multiobjective optimization problem denoted as (\ref{1}). The algorithm achieves a faster convergence rate while also gain the sequential convergence.

Subsequently, we present the methodology employed to address the optimization problem denoted as (\ref{1}). Similar to the exposition in \cite{tanabe2023accelerated}, a subproblem is delineated and resolved in each iteration. Using the descent lemma, the proposed approach tackles the ensuing subproblem for prescribed values of \(x \in \text{dom}(F)\), \(y \in \mathbb{R}^{n}\), and \(\ell \geq L\):
\begin{equation}\label{4}
	\min_{z \in \mathbb{R}^{n}} \varphi_{\ell}(z;x,y,\mu),
\end{equation}
where
\begin{equation}\label{5}
	\varphi_l(z;x,y,\mu) := \max_{i=1,\dotsb,m} \left[\left\langle \nabla \tilde{f}_i (y,\mu),z-y\right\rangle +g_i(z)+\tilde{f}_i(y,\mu)-\tilde{F}_i(x,\mu) \right]  + \frac{\ell }{2} \left\|z-y\right\|^2. \\
\end{equation}

Since $g_{i}$ is convex for all $i = 1,\dotsb,m,z \mapsto \varphi_{\ell}(z;x,y,\mu)$ is strongly convex.Thus,the subproblem (\ref{4}) has a unique optimal solution $p_{\ell}(x,y,\mu)$ and attain the optimal function value $\theta_{\ell}(x,y,\mu)$,i.e.,
\begin{equation}\label{6}
	p_{\ell}(x,y,\mu) := \arg \min_{z\in \mathbb{R}^{n}} \varphi_{\ell}(z,x,y,\mu) \  \text{and} \  \theta_{\ell}(x,y,\mu) := \min_{z\in \mathbb{R}^{n}} \varphi_{\ell}(z,x,y,\mu).
\end{equation}

Furthermore, the optimality condition associated with the optimization problem denoted as (\ref{4}) implies that, for all \(x \in \text{dom} \, F\) and \(y \in \mathbb{R}^{n}\), there exists \(\eta(x, y, \mu) \in \partial g(p_{\ell}(x, y, \mu))\) and a Lagrange multiplier \(\lambda(x, y) \in \mathbb{R}^{m}\) such that
\begin{equation}\label{7}
	\sum_{i=1}^{m} \lambda_{i}(x,y) [ \nabla \tilde{f}_{i}(y,\mu) + \eta_{i}(x,y,\mu)] = -\ell [p_{\ell}(x,y) - y]
\end{equation}
\begin{equation}\label{8}
	\lambda(x,y) \in \Delta^{m}, \quad \lambda_{j}(x,y) = 0 \quad \forall j \notin \mathcal{I}(x,y),
\end{equation}
where $\Delta^{m}$ denotes the standard simplex and
\begin{equation}
	\mathcal{I}(x,y) := \arg\max_{i = 1,\dotsb,m}[ \left\langle \nabla \tilde{f}_{i}(y,\mu) , p_{\ell}(x,y,\mu) - y \right\rangle + g_{i}(p_{\ell}(x,y,\mu)) +\tilde{f}_{i}(y,\mu) - \tilde{F}_{i}(x,\mu) ].
\end{equation}
Before we present the algorithm framework, we first give the following assumption.

\begin{assumption}\label{a1}
	Suppose $X^{*}$ is set of the weakly Pareto optimal points and $\mathcal{L}_{\tilde{F}}(c) := \{x \in \mathbb{R}^{n}|\tilde{F}(x) \leq c\}$, then for any $x \in \mathcal{L}_{\tilde{F}}(\tilde{F}(x^0))$, then there exists
	$x \in X^{*}$such that $\tilde{F}(x^{*}) \leq \tilde{F}(x)$and
	$$R := \sup_{\tilde{F}^{*} \in \tilde{F}(X^{*} \cap \mathcal{L}_{\tilde{F}}(\tilde{F}(x^0)))} \inf_{z \in \tilde{F}^{-1}(\{\tilde{F}^{*}\})} \left\|z - x^0\right\|^2 < + \infty.$$
\end{assumption}

For easy of reference and corresponding to its structure, we call the proposed
algorithm the smoothing accelerated proximal gradient method with
extrapolation term for nonsmooth multiobjective
optimization(SAPGM) in this paper.The algorithm is in the following form.
\begin{algorithm}
	\renewcommand{\algorithmicrequire}{\textbf{Input:}}
	\renewcommand{\algorithmicensure}{\textbf{Output:}}
	\caption{The Smoothing Accelerated Proximal Gradient Method with
		Extrapolation term for Non-smooth Multi-objective
		Optimization}
	\label{alg1}
	\begin{algorithmic}[1]
		\REQUIRE Take initial point $x^{-1}=x^0  \in \text{dom}F$, $y^{0} = x^{0}$, $l \geq \tilde{L}$, $\varepsilon >0$, $\mu_0 \in R_{++}$, $\gamma_0 \in R_{++}$. Choose parameters  $\eta \in (0,1)$, $\alpha >3$, $\sigma \in ( \frac{1}{2} , 1 ]$. Set $k=0$.
		\LOOP
		\STATE Compute $$y^{k} = x^{k} + \frac{k-1}{k + \alpha -1} (x^{k} - x^{k-1})$$
		$$\mu_{k+1} = \frac{\mu_0}{(k+ \alpha -1) \ln^{\sigma}(k+ \alpha -1)}$$
		\STATE Set $\bar{\gamma}_{k+1} = \gamma_k ,\ell = {\bar{\gamma}_{k+1} \mu_{k+1}}$,compute$$\hat{x}^{k+1} = p_{\ell}(x^k,y^{k},\mu^{k+1})$$
		\IF{$$
			\begin{aligned}
				2\min_{i}(\tilde{f}_{i}(\hat{x}_{k+1},\mu_{k+1})-\tilde{f}_{i}(y_{k},\mu_{k+1})&-\langle\nabla\tilde{f}_{i}(y_{k},\mu_{k+1}),\hat{x}_{k+1}-y_{k}\rangle ) \\
				&>\frac{1}{\bar{\gamma}_{k+1}\mu_{k+1}}\parallel \hat{x}_{k+1}-y_{k}\parallel^{2}
			\end{aligned}
			$$}
		\STATE $$\bar{\gamma}_{k+1} = \eta \bar{\gamma}_{k+1} \ and \ go \ to \ step \ 3$$
		\ELSE
		\STATE $$\gamma_{k+1} = \bar{\gamma}_{k+1},x^{k+1} = \hat{x}^{k+1} \  and \ go \ to \ step \ 2, k + 1$$
		\ENDIF
		
		\IF{$\left\|x^k - x^{k+1} \right\| < \varepsilon$ and $\mu_{k+1} < \epsilon$}
		\RETURN $x^{k+1}$
		\ENDIF
		\ENDLOOP
		\ENSURE  $x^*$: A weakly Pareto optimal point
	\end{algorithmic}
\end{algorithm}

\section{The convergence rate analysis of SAPGM}
\subsection{Some Basic Estimation}
This section shows that SAPGM has a faster convergence rate than $O(1/k^2)$ under the  Assumption (\ref{a1}). For the convenience of the complexity analysis, we use some functions defined in \cite{tanabe2023accelerated}.For $k \geq 0$,let $W_{k} : \mathbb{R}^{n} \to \mathbb{R} \cup \{ - \infty\}$ and $u_{k} : \mathbb{R}^{n} \to \mathbb{R}$ be defined by
\begin{equation}\label{wkuk}
	\begin{aligned}
		W_{k}(z,\mu_{k}) &:= \min_{i =1,\dotsb,m} [\tilde{F}_{i}(x^{k},\mu_{k}) - F_{i}(z)]+ \kappa \mu_{k}, \\
		u_{k} &:= \frac{k + \alpha - 1}{\alpha - 1}x^{k+1} - \frac{k}{\alpha - 1}x^{k}.
	\end{aligned}
\end{equation}
Given a fixed weakly Pareto solution $x^{*} \in \mathbb{R}^{n}$, define the global energy function which serves for Lyapunov analysis:
\begin{equation}\label{11}
	\begin{aligned}
		\mathscr{E}_{k+1} := \frac{2\gamma \mu_{k}}{\alpha -1} (k +\alpha -1)^2 W_{k} &+ (\alpha - 1) \left\|u^{k} - x^{*}\right\|^2 \\
		&+ (\frac{4\kappa \gamma_0  \mu_0}{2\sigma - 1})\mu_{k}(k+ \alpha -2) \ln^{1-\sigma}(k+\alpha-2),
	\end{aligned}	
\end{equation}
where $W_{k} := W_{k}(x^{*},\mu_k)$.

Following the properties outlined in \cite{wu2023smoothing}, we present the following properties regarding the sequence $\{\mathscr{E}_{k}\}$.

\begin{proposition}\label{prop5.2}
	Let $\mathscr{E}_{k}$ the sequence defined in (\ref{11}). Then, for any $k\geq1$, we have
	\begin{equation}\label{14}
		{\mathscr{E}}_{k+1}+\frac{2(\alpha-3)\gamma_{k+1}\mu_{k+1}}{\alpha-1}(k+\alpha-1)\:W_{k}\leq{\mathscr{E}}_{k}.
	\end{equation}
	Moreover,
	
	(i) the sequence $\mathscr{E}_{k}$ is non-increasing for all $ k\geq1,\text{and} \ \lim_{k\to\infty} \mathscr{E}_{k} \ exists;$
	
	(ii) for every $k \geq 1$
	$$
	\mathscr{E}_{k}\leq(\alpha-1)\|z-x^{0}\|^{2}+4(\alpha-1)\kappa\mu_0^2+\frac{4\kappa\gamma_0\mu_{0}^{2}}{2\sigma-1}(\alpha-1)\ln^{1-\sigma}(\alpha-1).
	$$
\end{proposition}

\begin{proof}
	Before proving the proposition, we should discuss the following inequality where is crucial for proving:
	For any $x \in \mathcal{X}$ and $k \in \mathbb{N}$,it holds that for all $i = 1,\dotsb,m$
	\begin{equation}\label{eq21}
		\begin{aligned}
					\tilde{F}_i(x^{k+1},\mu_{k+1}) &\leq \tilde{F}_i(x, \mu_{k+1}) + (\gamma_{k+1} \mu_{k+1}) ^{-1} \left\langle y^k - x^{k+1}, y^k - x \right\rangle \\
					&\quad- \frac{1}{2}  (\gamma_{k+1} \mu_{k+1}) ^{-1}  \parallel x^{k+1} - y^k \parallel^2.
		\end{aligned}
	\end{equation}
	
	From step 4 and step 7 of the algorithm, we can see that for all $i = 1, \dotsb,m$, the following inequality holds
	\begin{equation}\label{eq23}
		\begin{aligned}
			\tilde{f}_{i}(x^{k+1},\mu_{k+1})  \leq& \tilde{f}_{i}(y^{k},\mu_{k+1}) + \left\langle\nabla\tilde{f}_{i}(y^{k},\mu_{k+1}),x^{k+1}-y^{k}\right\rangle \\
			&+ \frac{1}{2\gamma_{k+1}\mu_{k+1}}\parallel x^{k+1}-y^{k}\parallel^{2}.
		\end{aligned}	
	\end{equation}
	Set
	$$\begin{aligned}
		Q(x,y,\mu,\gamma) :=  \tilde{f}_{i}(y,\mu) + \left\langle\nabla\tilde{f}_{i}(y,\mu),x-y\right\rangle
		+ \frac{1}{2\gamma\mu}\parallel x-y\parallel^{2} + g_i(x).
	\end{aligned}$$
	We noticed that for a fixed $y, \mu $ and $\gamma $, function $Q (x, y, \mu, \gamma) $in $(\gamma\mu) ^ {-1} $as coefficient of strong convex function. Therefore, $Q (x, y, \mu, \gamma) $  has a unique global minimum on $\mathcal{X}$, its record of $p (y, \mu, \gamma) $, namely:
	$$p(y,\mu,\gamma) := \arg\min_{x \in \mathcal{X}} Q(x,y,\mu,\gamma).$$
	Since $p(y,\mu,\gamma)$ is the minimum, we infer that
	\begin{equation}\label{eq19}
		Q(x,y,\mu,\gamma) \geq Q(p(y,\mu,\gamma), y, \mu, \gamma) + \frac{1}{2} (\gamma\mu)^{-1} \parallel x - p(y, \mu, \gamma) \parallel^2, \forall x
		\in \mathcal{X}.
	\end{equation}
	Combining Step 3  and Step 7 with the definition of the proximal operator, we get
	\begin{equation}\label{eq20}
		x^{k+1} = p (y^k,\mu_{k+1},\gamma_{k+1}).
	\end{equation}
	Taking $y = y^k, \mu = \mu_{k+1}$ and $\gamma = \gamma_{k+1}$ in inequality (\ref{eq19}), we have
	$$
	Q(x,y^{k}, \mu_{k+1}, \gamma_{k+1}) \geq Q(x^{k+1}, y^k,\mu_{k+1}, \gamma_{k+1}) + \frac{1}{2} (\gamma_{k+1} \mu_{k+1})^{-1} \parallel x - x^{k+1} \parallel^2, \forall x \in \mathcal{X}.
	$$
	After some rearrangement, we deduce that, for any $x \in \mathcal{X}$,
	\begin{equation}\label{eq22}
		\begin{aligned}
			g_i(x^{k+1}) \leq g_i(x) + \left\langle \nabla \tilde{f}_i(y^k,\mu_{k+1}),x-x^{k+1} \right\rangle + \frac{1}{2} (\gamma_{k+1} \mu_{k+1})^{-1} \parallel x - y^k \parallel^2 \\
			- \frac{1}{2} (\gamma_{k+1} \mu_{k+1})^{-1} \parallel x - x^{k+1} \parallel^2 - \frac{1}{2} (\gamma_{k+1} \mu_{k+1})^{-1} \parallel x^{k+1} - y^k \parallel^2.
		\end{aligned}
	\end{equation}
	Adding (\ref{eq22}) and (\ref{eq23}), and using the convexity of $\tilde{f}_i$, we deduced that, for any $x \in \mathcal{X}$,
	\begin{equation}\label{eq24}
		\begin{aligned}
			\tilde{F}_i(x^{k+1}, \mu_{k+1}) &= \tilde{f}_i(x^{k+1}, \mu_{k+1}) + g_i(x^{k+1}) \\
			&\leq \tilde{F}_i(x, \mu_{k+1}) + \frac{1}{2} (\gamma_{k+1} \mu_{k+1})^{-1} \parallel x - y^{k} \parallel^2 \\
			& \quad - \frac{1}{2} (\gamma_{k+1} \mu_{k+1})^{-1} \parallel x - x^{k+1} \parallel^2.
		\end{aligned}
	\end{equation}
	So we have proven that inequality (\ref{eq21}) holds.
	
	Recalling that $W_k$ define in \cite{wu2023smoothing}, in order to be exactly, we use $W^{'}_k$:
	$$
	W^{'}_k = \tilde{F}(x^k,\mu_k) + \kappa \mu_{k} - F(x^{*}),
	$$
	where $x^{*} \in \arg\min F(x)$. We can see that if we just let $F$ be replaced by $F_i$,let $x^{*}$ be the weak Pareto solution of $\min F(x)$, then we get
	$$
	W^{'}_{k,i} = \tilde{F}_i(x^k,\mu_k) + \kappa \mu_{k} - F_i(x^{*}).
	$$
	Combining (\ref{eq21}) and this definition, the $W_k$ define in our article has following relation with $W^{'}_{k,i}$:
	$$
	W_k = \min_{i = 1,\dotsb,m} W^{'}_{k,i} \leq W^{'}_{k,i}, \forall i =1,\dotsb,m.
	$$
	So we can get the following two inequalities of $W_{k}$ and $W_{k+1}$, which are basic for this discussion:
	\begin{equation}\label{15}
		W_{k+1}\leq W_{k}+(\mu_{k+1}\gamma_{k+1})^{-1}\langle y^{k}-x^{k+1},y^{k}-x^{k}\rangle-\frac{1}{2}(\gamma_{k+1}\mu_{k+1})^{-1}\|x^{k+1}-y^{k}\|^{2},
	\end{equation}
	and
	
	\begin{equation}\label{16}
		\begin{aligned}
					W_{k+1}\leq& (\mu_{k+1} \gamma_{k+1})^{-1}\langle y^k-x^{k+1},y^k-z\rangle \\ &-\frac{1}{2}(\mu_{k+1}\gamma_{k+1})^{-1}\|x^{k+1}-y^k\|^2+2\kappa\mu_{k+1}.
		\end{aligned}
	\end{equation}
	
	The rest of the proof is similar to the Proposition 3.1 proof in the article\cite{wu2023smoothing}, so we don't want to go into details.
\end{proof}

As a result of Proposition (\ref{prop5.2}), we obtain some important properties of $W_k$ as shown below, where we need to introduce an important lemma on sequence convergence.

\begin{lemma}[\cite{wu2023smoothing} Lemma 3.3]\label{lem5.3}
	Let $\{a_k\}$ be a sequence of nonnegative numbers, and satisfy
	
	$$
	\sum_{k=1}^\infty(a_{k+1}-a_k)_+<\infty.
	$$
	
	Then, $\lim_{k\to\infty}a_k$ exists.
\end{lemma}

\begin{lemma}
	$\{\gamma_k\}$ is non-increasing.
\end{lemma}
\begin{proof}
	The iterative format of $\gamma_k$ shows that it is non-increasing,
\end{proof}
\begin{theorem}\label{th5.4}
	Suppose$\left\{x^k\right\}and\left\{y^k\right\}$be the sequences generated by SAPGM, for any $z\in\mathbb{R}^n\:and\:\alpha>3$,it holds that
	
	(i) $\sum_{k=1}^{\infty} \mu_{k}\gamma_{k}(k+\alpha-2)W_k < \infty$;
	
	(ii)$\lim\limits_{k \to \infty}[(k-1)^2\left\|x^k - x^{k-1}\right\|^2 + 2\mu_{k}\gamma_{k}(k+\alpha-2)^2W_k]$ exists;
	
	(iii)$\sum_{k=1}^{\infty}(k-1)\left\|x^k -x^{k-1}\right\|^2 < \infty$;
	
	(iv)$u_0(x^k)=o(\ln^\sigma k/k).$
\end{theorem}
\begin{proof}
	Before proving, we state that the proof of (i),(ii), and (iii) are similar to the proof of Proposition 3.2 in \cite{wu2023smoothing}.
	
	(i) By summing inequality (\ref{14}) from $k = 1$ to $K$, we obtain:
	$$
	\delta_{K+1} + \frac{2(\alpha - 3)}{\alpha - 1} \sum_{k=1}^{K} \gamma_{k+1} \mu_{k+1} (k + \alpha - 1) W_k \leq \mathscr{E}_1.
	$$
	Now, letting $K \to \infty$ in the above inequality and using Proposition (ii), since $\alpha > 3$ and $\mathscr{E}_k \geq 0$ for all $k \geq 0$, we can infer that
	\begin{equation}\label{4.13}
		\sum_{k=1}^{\infty} \mu_{k+1} \gamma_{k+1} (k + \alpha - 1) W_k \leq \frac{(\alpha - 1)\mathscr{E}_1}{2(\alpha - 3)} < \infty.
	\end{equation}

	Since for all $k \geq 1$, it holds that
	$$
	\mu_{k+1} \gamma_{k+1} (k + \alpha - 1) = \frac{\mu_0 \gamma_0}{\ln^{\sigma}(k + \alpha - 1)} \geq \mu_k \gamma_k \frac{(k + \alpha - 2) \ln^{\sigma}(k + \alpha - 2)}{\ln^{\sigma}(k + \alpha - 1)}.
	$$
	We further obtain:
	$$
	\mu_{k+1} \gamma_{k+1} (k + \alpha - 1) \geq \frac{\ln^{\sigma}(\alpha - 1)}{\ln^{\sigma} \alpha} \mu_k \gamma_k (k + \alpha - 2).
	$$
	This inequality follows from the fact that $\frac{\ln^{\alpha}(k + \alpha - 2)}{\ln^{\alpha}(k + \alpha - 1)}$ is increasing for all $k \geq 1$.
	
	Therefore, inequality (\ref{4.13}) implies:
	$$
	\sum_{k=1}^{\infty} \mu_k \gamma_k (k + \alpha - 2) W_k < \infty.
	$$

	(ii) Returning to inequality (\ref{15}) and using the identity
$$
-\|a - b\|^2 + 2 \langle b - a, b - c \rangle = -\|a - c\|^2 + \|b - c\|^2
$$
with $ a = x^{k+1}, b = y^k $, and $ c = x^k $, we deduce:
$$
W_{k+1} \leq W_k - \frac{1}{2} (\mu_{k+1} \gamma_{k+1})^{-1} \|x^{k+1} - x^k\|^2 + \frac{1}{2} (\mu_{k+1} \gamma_{k+1})^{-1} \|y^k - x^k\|^2.
$$
By the definition of $ y^k $ in SAPGM and multiplying the inequality by $ \mu_{k+1} \gamma_{k+1} (k + \alpha - 1)^2 $, we get:
$$
\begin{aligned}
	2 \mu_{k+1} \gamma_{k+1} &(k + \alpha - 1)^2 W_{k+1} + (k + \alpha - 1)^2 \|x^{k+1} - x^k\|^2 \\
	&\leq 2 \mu_{k+1} \gamma_{k+1} (k + \alpha - 1)^2 W_k + (k - 1)^2 \|x^k - x^{k-1}\|^2.
\end{aligned}
$$
Since $ k + \alpha - 1 \geq k $, we can rearrange terms to obtain:
\begin{equation}\label{4.16}
	0 \geq k^2 \|x^{k+1} - x^k\|^2 - (k-1)^2 \|x^k - x^{k-1}\|^2 + 2 \mu_{k+1} \gamma_{k+1} (k + \alpha - 1)^2 (W_{k+1} - W_k).
\end{equation}
Next, observe that
$$
\begin{aligned}
	&\mu_{k+1} (k + \alpha - 1)^2 - \mu_k (k + \alpha - 2)^2  \\
	=& \mu_{k+1} (k + \alpha - 1) \left( k + \alpha - 1 - \frac{(k + \alpha - 2) \ln^\sigma(k + \alpha - 1)}{\ln^\sigma(k + \alpha - 2)} \right) \\
	\leq& \mu_{k+1} (k + \alpha - 1),
\end{aligned}
$$
which leads to:
\begin{equation}\label{4.17}
	\begin{aligned}
		&\mu_{k+1} (k + \alpha - 1)^2 W_{k+1} - \mu_k (k + \alpha - 2)^2 W_k \\
		=& \mu_{k+1} (k + \alpha - 1)^2 (W_{k+1} - W_k) + \left( \mu_{k+1} (k + \alpha - 1)^2 - \mu_k (k + \alpha - 2)^2 \right) W_k \\
		\leq& \mu_{k+1} (k + \alpha - 1)^2 (W_{k+1} - W_k) + \mu_{k+1} (k + \alpha - 1) W_k.
	\end{aligned}
\end{equation}

For simplicity, define:
$$
\zeta_k := (k - 1)^2 \|x^k - x^{k-1}\|^2 + 2 \mu_{k} \gamma_{k} (k + \alpha - 2)^2 W_k.
$$
Substituting (\ref{4.17}) into (\ref{4.16}), we obtain:
\begin{equation}\label{4.19}
	\zeta_{k+1} - \zeta_k \leq 2 \gamma_{k+1} \mu_{k+1} (k + \alpha - 1) W_k.
\end{equation}
Taking the positive part of the left-hand side and using inequality (\ref{4.13}), we find:
$$
\sum_{k=1}^{\infty} (\zeta_{k+1} - \zeta_k)_+ < \infty.
$$
Since $ \zeta_k \geq 0 $, by Lemma \ref{lem5.3}, we infer that $ \lim_{k \to \infty} \zeta_k $ exists.

	(iii) In view of $\alpha>3$, we observe that
	
	$$
	\begin{aligned}
		&(k+\alpha-1)^2\|x^{k+1}-x^k\|^2-(k-1)^2\|x^k-x^{k-1}\|^2\\
		\geq&(k+2)^2\|x^{k+1}-x^k\|^2-(k-1)^2\|x^k-x^{k-1}\|^2\\
		\geq& k^2\|x^{k+1}-x^k\|^2-(k-1)^2\|x^k-x^{k-1}\|^2+4k\|x^{k+1}-x^k\|^2,\end{aligned}
	$$
	combining which with (\ref{4.19}), then we obtain
	
	$$
	\begin{aligned}
		&k^2\|x^{k+1}-x^k\|^2-(k-1)^2\|x^k-x^{k-1}\|^2+4k\|x^{k+1}-x^k\|^2\\
		\leq& 2\mu_{k}\gamma_{k}(k+\alpha-2)^{2}W_k-2\mu_{k+1}\gamma_{k+1}(k+\alpha-1)^{2}W_{k+1}+2\mu_{k+1}\gamma_{k+1}(k+\alpha-1)W_k.
	\end{aligned}
	$$
	
	Summing up the above inequality for $k=1,2,\ldots,K$, we obtain
	
	\begin{equation}\label{36}
		\begin{aligned}
			&K^{2}\|x^{K+1}-x^{K}\|^{2}+4\sum_{k=1}^{K}k\|x^{k+1}-x^{k}\|^{2}\\
			\leq& 2\mu_{1}\gamma_{1}(\alpha-1)^{2}W_{1}+2\sum_{k=1}^{K}\mu_{k+1}\gamma_{k+1}(k+\alpha-1)W_{k}.
		\end{aligned}
	\end{equation}
	
	Since $W_1\leq \mathscr{E}_1$, letting $K$ tend to infinity in the above inequality, by (\ref{4.13}) and (\ref{36}), we  have $$\sum_{k=1}^\infty k\|x^{k+1}-x^k\|^2<\infty.$$
	
%
%
	
	(iv) Combining (i) and (ii), we have
	$$\sum_{k=1}^{\infty}\left[(k-1)\|x^{k}-x^{k-1}\|^{2}+2\gamma_{k}\mu_{k}(k+\alpha-2)W_{k}\right]<\infty,$$
	which implies
	$$
	\begin{aligned}
		&\sum_{k=1}^{\infty}\frac{\ln(k+\alpha-2)}{(k+\alpha-2)\ln(k+\alpha-2)}\left[(k-1)^{2}\|x^{k}-x^{k-1}\|^{2}+2\gamma_{k}\mu_{k}(k+\alpha-2)^{2}W_{k}\right]\\
		&\leq\sum_{k=1}^{\infty}\frac{1}{k+\alpha-2}\left[(k-1)(k+\alpha-2)\|x^{k}-x^{k-1}\|^{2}+2\gamma_{k}\mu_{k}(k+\alpha-2)^{2}W_{k}\right]\\
		&=\sum_{k=1}^{\infty}\left[(k-1)\|x^{k}-x^{k-1}\|^{2}+2\gamma_{k}\mu_{k}(k+\alpha-2)W_{k}\right]<\infty.
	\end{aligned}
	$$
	Observe that $\sum_{k=1}^{\infty}\frac1{(k+\alpha-1)\ln(k+\alpha-1)}=\infty$, then
	\begin{equation}\label{49}
		\lim_{k\to\infty}\inf\ln(k+\alpha-2)\left((k-1)^2\|x^k-x^{k-1}\|^2+2\gamma_k\mu_k(k+\alpha-2)^2W_k\right)=0.
	\end{equation}
	Combining this with (ii), we obtain
	\begin{equation}\label{50}
		\lim_{k\to\infty}(k-1)^2\|x^k-x^{k-1}\|^2+2\gamma_k\mu_k(k+\alpha-2)^2W_k=0,
	\end{equation}
	by $W_k\geq0$, which further implies
	\begin{equation}\label{51}
		\lim_{k\to\infty}(k-1)^2\|x^k-x^{k-1}\|^2=0\quad\mathrm{and}\quad\lim_{k\to\infty}\gamma_k\mu_k(k+\alpha-2)^2W_k=0.
	\end{equation}
	Recalling the definition of $\mu_k$ in step 2 of algorithm,the non-increment of $\gamma$ and (\ref{wkuk}), the second equation in (\ref{51})
	implies 
	$$
		\lim\limits_{k \to \infty} (k + \alpha - 2) \ln^\sigma (k + \alpha - 1)W_k = 0
	$$
	By the definition of $W_k = W_k(x^*,\mu_k)$,we get that
	$$
		\lim\limits_{k \to \infty} (k + \alpha - 2) \ln^\sigma (k + \alpha - 1) \left(\min_{i = 1,\dotsb,m}[ \tilde{F}_i(x^k,\mu_k) - F_i(x^*)] + \kappa \mu_k\right) = 0
	$$
	From the definition of $u_0$ and the fact that $x^*$ is a weak Pareto point, we infer that
	$$
	\begin{aligned}
		\min_{i = 1,\dotsb,m} [F_i(x^*) - F_i(z)] &\leq \min_{i = 1,\dotsb,m} [F_i(x_k) - F_i(z)] \\
		&= \min_{i = 1,\dotsb,m}[ F_i(x_k) - \tilde{F}_i(x_k,\mu_k) + \\
		&\quad\tilde{F}_i(x_k,\mu_k) - F_i(x^*) + F_i(x^*) - F_i(z)] \\
		&\leq \kappa \mu_k + \min_{i = 1,\dotsb,m} [\tilde{F}_i(x_k,\mu_k) - F_i(x^*)].
	\end{aligned}
	$$
	So we have
	$$
	\sup_{z \in \mathbb{R}^n}\min_{i = 1,\dotsb,m} [F_i(x^*) - F_i(z)] \leq \kappa \mu_k + \min_{i = 1,\dotsb,m} [\tilde{F}_i(x_k,\mu_k) - F_i(x^*)].$$
	Because $(k + \alpha - 2) \ln^\sigma (k + \alpha - 1) >0$, we get that
	$$
	\begin{aligned}
		0 &\leq (k + \alpha - 2) \ln^\sigma (k + \alpha - 1)\sup_{z \in \mathbb{R}^n}\min_{i = 1,\dotsb,m} [F_i(x^*) - F_i(z)] \\
		&\leq \kappa \mu_k + \min_{i = 1,\dotsb,m} [\tilde{F}_i(x_k,\mu_k) - F_i(x^*)].
	\end{aligned}
	$$
	So we know that
	$$
	\lim\limits_{k \to \infty} (k + \alpha - 2) \ln^\sigma (k + \alpha - 1)\left(\sup_{z \in \mathbb{R}^n}\min_{i = 1,\dotsb,m} [F_i(x^*) - F_i(z)]\right) = 0.
	$$
	This result illustrates that for any $\sigma\in(\frac12,1]$ in the SAPGM algorithm, it holds $u_0(x^k)=o(\ln^\sigma k/k).$
\end{proof}

\subsection{Sequential Convergence}
In this subsection, we are ready to analyze the convergence of the iterates generated by the SAPGM. In this context, we articulate the discrete manifestation of Opial's lemma, laying the groundwork for a rigorous examination of the convergence properties inherent in the sequence  $\{x^k\}.$

\begin{lemma}[\cite{wu2023smoothing} Lemma 3.4]\label{l5.6}
	Let $S$ be a nonempty subset of $\mathbb{R}^n$ and $\{z_k\}$ be a sequence of $\mathbb{R}^n.$
	
	Assume that
	
	(i) $\lim_{k\to\infty}\|z_k-z\|$ exists for every  $z \in S;$
	
	(ii) every sequential limit point of sequence $\{z_k\}$ as $k\to\infty$ belongs to S.
	
	Then, as $k\to\infty$, $\{z_k\}$ converge to a point in $S$.
\end{lemma}

To prove the sequential convergence, we must recall the following inequality on nonnegative sequences, which will be used in the forthcoming sequential convergence result.

\begin{lemma}[\cite{wu2023smoothing} Lemma 3.5]\label{l5.7}
	Assume $\alpha \geq 3$.Let $\{a_k\}$ and $\{\omega_k\}$ be two sequences of nonnegative numbers such that
	
	$$
	a_{k+1}\leq{\frac{k-1}{k+\alpha-1}}a_{k}+\omega_{k}
	$$
	for all $k\geq1$.If $\sum_{k=1}^\infty k\omega_k<\infty$,then $\sum_{k=1}^\infty a_k<\infty.$
\end{lemma}

\begin{theorem}
	Let $\{x_k\}$ be the sequence generated by the algorithm. Then, as $k \to \infty$, the sequence $\{x_k\}$ converges to a weak Pareto solution of the original problem.
\end{theorem}

\begin{proof}
	Let $\{x_k\}$ be the sequence generated by SAPGM, and let $\overline{x}$ be its cluster point. For any $z \in \mathbb{R}^n$, if we can prove that $\mu_0(\overline{x}) = 0$, then $\overline{x}$ is a weak Pareto optimal solution of the original problem.
	
	Because $\mu_0(\overline{x}) = 0 \iff \max_{i = 1,\dots,m}[F_i(z) - F_i(\overline{x})] \geq 0, \forall z \in \mathbb{R}^n$, we only need to prove
	$$\max_{i = 1,\dots,m}[F_i(z) - F_i(\overline{x})] \geq 0, \forall z \in \mathbb{R}^n.$$
	
	Therefore, we can reform the problem by some smoothing function properties in \cite{chen2012smoothing}.
	$$
	\begin{aligned}
		\max_{i = 1,\dots,m}[F_i(z) - F_i(\overline{x})] =& \max_{i = 1,\dots,m}[F_i(z) - \tilde{F}_i(z,\mu_{k}) + \tilde{F}_i(z,\mu_{k}) - \tilde{F}_i(\overline{x},\mu_{k}) \\
		&+ \tilde{F}_i(\overline{x},\mu_{k}) - F_i(\overline{x})] \\
		&\geq \max_{i = 1,\dots,m}[\tilde{F}_i(z,\mu_{k}) - \tilde{F}_i(\overline{x},\mu_{k})] - \kappa\mu_k.
	\end{aligned}
	$$
	
	Through the subproblem $\varphi_{\ell}$, we can get
	$$
	\begin{aligned}
		&\max_{i = 1,\dots,m}[\tilde{F}_i(\overline{x},\mu_k) - \tilde{F}_i(x,\mu_k)] \leq \varphi_{\ell}(\overline{x} + \alpha(z - \overline{x}); x, \overline{x}, \mu_k) \\
		&= \max_{i = 1,\dots,m} \left[\left\langle \nabla \tilde{f}_i(\overline{x}, \mu_k), \alpha(z - \overline{x}) \right\rangle + g_i(\overline{x} + \alpha(z - \overline{x})) + \tilde{f}_i(\overline{x}, \mu_k) - \tilde{F}_i(x, \mu_k) \right] \\
		&\quad + \frac{ (\gamma\mu_k)^{-1}}{2} \left\|\alpha(z - \overline{x})\right\|^2.
	\end{aligned}
	$$
	
	Due to the convexity of $\tilde{f}_i$, we have
	$$
	\max_{i = 1,\dots,m}[\tilde{F}_i(\overline{x},\mu_k) - \tilde{F}_i(x,\mu_k)] \leq \max_{i = 1,\dots,m}[\tilde{F}_i(\overline{x} + \alpha(z - \overline{x}), \mu_k) - \tilde{F}_i(x, \mu_k)] + \frac{(\gamma\mu_k)^{-1}}{2} \left\|\alpha(z - \overline{x})\right\|^2.
	$$
	
	Furthermore, the convexity of $\tilde{F}_i$ leads to
	$$
	\begin{aligned}
		\max_{i = 1,\dots,m}[\tilde{F}_i(\overline{x}, \mu_k) - \tilde{F}_i(x, \mu_k)] &\leq \max_{i = 1,\dots,m}[\alpha \tilde{F}_i(z, \mu_k) + (1 - \alpha) \tilde{F}_i(\overline{x}, \mu_k) - \tilde{F}_i(x, \mu_k)] \\
		&\quad+ \frac{(\gamma\mu_k)^{-1}}{2} \left\|\alpha(z - \overline{x})\right\|^2 \\
		&\leq \alpha \max_{i = 1,\dots,m}[\tilde{F}_i(z, \mu_k) - \tilde{F}_i(\overline{x}, \mu_k)] + \max_{i = 1,\dots,m}[\tilde{F}_i(\overline{x}, \mu_k) - \tilde{F}_i(x, \mu_k)] \\
		&\quad + \frac{(\gamma\mu_k)^{-1}}{2} \left\|\alpha(z - \overline{x})\right\|^2.
	\end{aligned}
	$$
	
	Therefore, we gain
	$$
	\max_{i = 1,\dots,m}[\tilde{F}_i(z, \mu_k) - \tilde{F}_i(\overline{x}, \mu_k)] \geq -\frac{\alpha (\gamma\mu_k)^{-1}}{2} \left\|z - \overline{x}\right\|^2.
	$$
	Letting $\alpha$ tend to 0 monotonically, we get $\max_{i = 1,\dots,m}[\tilde{F}_i(z, \mu_k) - \tilde{F}_i(\overline{x}, \mu_k)] \geq 0.$
	
	At the same time, letting $k \to \infty$, we have
	$$\max_{i = 1,\dots,m}[F_i(z) - F_i(\overline{x})] \geq 0, \forall z \in \mathbb{R}^n.$$
	
	Thus, $\overline{x}$ is a weak Pareto optimal solution of the original problem.
	
	Next, if we prove that for all weak Pareto optimal solution $\overline{x}$, $\lim_{k \to \infty} \|x_k - \overline{x}\|$ exists, then we can deduce the convergence of the sequence $\{x_k\}$ from the lemma.
	
	Because $W_{k+1}(\overline{x}) \geq 0$ and following inequality
	$$
	\begin{aligned}
		W_{k+1}(z, \mu_{k+1}) &\leq \frac{-\mu_{k+1}\gamma_{k+1}}{2} \{2\langle x^{k+1} - y^{k+1}, y^{k+1} - z \rangle + \|x^{k+1} - y^{k+1}\|^2\} +2\kappa \mu_{k+1},
	\end{aligned}
	$$
	we get
	$$
	2 \langle y^{k+1} - x^{k+1}, y^{k+1} - \overline{x}\rangle - \|x^{k+1} - y^{k+1}\|^2 + 2\kappa \mu_{k+1} \geq 0,
	$$
	which implies
	$$
	\|y^{k+1} - \overline{x}\|^2 - \|x^{k+1} - \overline{x}\|^2 +2 \kappa \mu_{k+1}\geq 0.
	$$
	
	Then,
	$$
	\begin{aligned}
		\|x^{k+1} - \overline{x}\|^2 &\leq \|y^{k+1} - \overline{x}\|^2 \\
		&= \|x^k + \frac{k-1}{k+\alpha-1}(x^k - x^{k-1}) - \overline{x}\|^2 +2 \kappa \mu_{k+1}\\
		&= \|x^{k} - \overline{x}\|^2 + \left(\frac{k-1}{k+\alpha-1}\right)^2 \|x^{k} - x^{k-1}\|^2 \\
		&\quad+ 2\left(\frac{k-1}{k+\alpha-1}\right)\langle x^{k} - \overline{x}, x^{k} - x^{k-1}\rangle +2 \kappa \mu_{k+1}\\
		&= \|x^{k} - \overline{x}\|^2 + \left(\left(\frac{k-1}{k+\alpha-1}\right)^2 + \frac{k-1}{k+\alpha-1}\right) \|x^{k} - x^{k-1}\|^2 \\
		&\quad + \frac{k-1}{k+\alpha-1}(\|x^{k} - \overline{x}\|^2 - \|x^{k-1} - \overline{x}\|^2) +2 \kappa \mu_{k+1}\\
		&\leq \|x^{k} - \overline{x}\|^2 + 2 \|x^{k} - x^{k-1}\|^2 \\
		&\quad+ \frac{k-1}{k+\alpha-1}(\|x^{k} - \overline{x}\|^2 - \|x^{k-1} - \overline{x}\|^2)+2 \kappa \mu_{k+1}.
	\end{aligned}
	$$
	
	Let $h_k := \|x^k - \overline{x} \|^2$, we have
	$$
	(h_{k+1} - h_{k})_+ \leq \left(\frac{k-1}{k+\alpha-1}\right)(h_{k} - h_{k-1})_+ + 2 \|x^{k} - x^{k-1}\|^2 +2 \kappa \mu_{k+1}.
	$$
	
	From Lemmas \ref{l5.6} and \ref{l5.7}, we obtain $\sum_{k=1}^\infty (h_{k+1} - h_k)_+ < \infty$, and from the non-negativity of $\{h_k\}$, we know that $\lim_{k \to \infty} h_k$ exists.
\end{proof}

\begin{remark}
	Now, from the sequence convergence, we set
	$$\parallel x^k - x^{k+1} \parallel_{\infty} < \epsilon, and \ \mu_{k + 1} < \epsilon$$
	as the algorithm-stopping criterion. From the above proof process, it is very natural to see the reason for our setting.
\end{remark}

\section{Efficient computation of the subproblem via its dual}
In the previous section, we proved the global convergence and complexity results of SAPGM. Subsequently, our focus shifts to empirically assessing the method's practical efficacy. Specifically, we elucidate a methodology for computing the subproblem. To commence, let us introduce a formal definition.
\begin{equation}\label{38}
	\psi_{i} (z; x, y,\mu) := \left\langle \nabla \tilde{f}_{i} (y,\mu), z - y \right\rangle + g_{i}(z) + \tilde{f}_{i}(y,\mu) - \tilde{F}_{i}(x,\mu) + \frac{\ell}{2} \left\|z - y\right\|^2
\end{equation}
for all $i = 1,\dotsb, m$. Then, fixing some $\ell \geq L$, we can rewrite the objective function $\varphi_{\ell}(z;x,y)$ as
$$\varphi_{\ell}(z;x,y,\mu) = \max_{i = 1, \dotsb, m} \psi_{i} (z;x,y,\mu).$$

Based on the discussion in \cite{tanabe2023accelerated}, we obtain the dual problem as follows:
\begin{equation}\label{42}
	\begin{aligned}
		\max_{\lambda \in \mathbb{R}^m}& \quad \omega(\lambda) \\
		\mathrm{s.t.}&\quad\lambda\geq0\quad\mathrm{and}\quad\sum_{i=1}^{m}\lambda_{i}=1,
	\end{aligned}
\end{equation}

where
\begin{equation}
	\begin{aligned}
		\omega(\lambda)&:=\ell {\mathcal{M}_{\frac{1}{\ell} \sum_{i=1}^{m} \lambda_{i} g_{i}}} \left(y - \frac{1}{\ell} \sum_{i=1}^{m} \lambda_{i} \nabla \tilde{f}_{i}(y, \mu)\right) \\
		&\quad-\frac{1}{2 \ell} \left\| \sum_{i=1}^{m} \lambda_{i} \nabla \tilde{f}_{i}(y,\mu)\right\|^2 + \sum_{i=1}^{m} \lambda_{i} \{\tilde{f}_{i}(\mu) - \tilde{F}_{i}(x,\mu)\}
	\end{aligned}
\end{equation}
Given the identification of the global optimal solution \(\lambda^*\) for the dual problem (\ref{42}), it becomes feasible to construct the optimal solution \(z^{*}\) for the original subproblem as follows:
\begin{equation}
	z^*=\mathbf{prox}_{\frac{1}{\ell}\sum_{i=1}^m\lambda_i^*g_i}\left(y-\frac{1}{\ell}\sum_{i=1}^m\lambda_i^*\nabla \tilde{f}_i(y,\mu)\right),
\end{equation}
where prox denotes the proximal operator.

So we can choose the Frank-Wolfe method \cite{jaggi2013frankwolfe} to solve the above dual problem (\ref{42}).

\begin{algorithm}
	\renewcommand{\algorithmicrequire}{\textbf{Input:}}
	\renewcommand{\algorithmicensure}{\textbf{Output:}}
	\caption{Frank-Wolfe algorithm}
	\label{alg2}
	\begin{algorithmic}[1]
		\REQUIRE $x \in D$,where $D$ is the feasible set of problem.$K$ is max iteration number,$\mu \in \mathbb{R}$ is smoothing parameter.
		\FOR{$k = 0,1,\dotsb,K$}
		\STATE Compute $s = \arg\min_{s \in D} \left\langle s,\nabla \tilde{F}(x^k, \mu) \right\rangle$
		\STATE Update $x^{k+1} := (1- \frac{2}{k+2}) x^{k} + \frac{2}{k+2} s$
		\ENDFOR
	\end{algorithmic}
\end{algorithm}

Additionally, $\omega$ is differentiable, which can make the Frank-Wolfe method easy to achieve, as the following Lemma shows.
\begin{lemma}[\cite{tanabe2023accelerated},Theorem 6.1]
	The function $\omega:\mathbb{R}^m\to\mathbb{R}$ defined by $\left(31\right)$ is continaously differentiahle at every $\lambda\in\mathbb{R}^m$ and
	$$\begin{aligned}
		\nabla\omega(\lambda)=&g\left(\textbf{prox}_{\frac{1}{\ell}\sum\limits_{i=1}^m\lambda_ig_i}\left(y-\frac1\ell\sum\limits_{i=1}^m\lambda_i\nabla \tilde{f}_i(y,\mu)\right)\right)\\
		+&J_{\tilde{f}}(y)\left(\textbf{prox}_{\frac1\ell\sum\limits_{i=1}^m\lambda_ig_i}\left(y-\frac1\ell\sum\limits_{i=1}^m\lambda_i\nabla \tilde{f}_i(y,\mu)\right)-y\right)+\tilde{f}(y,\mu)-\tilde{F}(x,\mu),
	\end{aligned}$$
	where prox is the proximal operator, and $J_{\tilde{c}}(y)$ is the Jacobian matrix at $y$ given by
	
	$$
	J_{\tilde{f}}(y):=\left(\nabla \tilde{f}_{1}(y,\mu),\ldots,\nabla \tilde{f}_{m}(y,\mu)\right)^{\top}.
	$$
	
\end{lemma}

The proof is similar to that in \cite{tanabe2023accelerated}. This theorem establishes that the dual problem denoted as (\ref{42}) constitutes an \(m\)-dimensional differentiable convex optimization problem. Consequently, the effective computation of the proximal operator for the summation \(\sum_{i=1}^{m} \lambda_{i}g_i\) in a rapid manner would enable the resolution of (\ref{42}) through the application of convex optimization techniques.

\section{Numerical experiments}
In this section, we present numerical results to show the good performance of the SAPGM algorithm for solving (\ref{1}). The numerical experiments are performed in Python 3.10 on a 64-bit Lenovo PC with a 12th Gen Intel(R) Core(TM) i7-12700H CPU @ 2.70 GHz and 16GB RAM. To compare with the SAPGM, we use DNNM \cite{gebken2021efficient}, the descent method for local Lipschitz multi-objective optimization problems, to conduct controlled experiments on the same test problems. For simplicity, we use Iter to represent the number of iterations and Time to represent the amount of time a program takes to run.

For convenience, we introduce some smoothing functions as follows:
For the maximum function $\max(z,0)$,we use its smoothing function \cite{feng2008smooth} as follow:
\begin{align*}
	\tilde{\phi}(z,\mu) =
	\begin{cases}
		0, &z< - \mu \\
		\frac{(z+\mu)^3}{6 \mu^2}, & - \mu \leq z < 0 \\
		z+\frac{(z+\mu)^3}{6 \mu^2}, & 0 \leq z \leq mu \\
		z, & z > \mu
	\end{cases}
\end{align*}
For the maximum function $\max(z_1,\dotsb,z_n)$,it can be represented by $\max\{z,0\}$ becase $\max\{a,b\} = a + \max\{b-a,0\}$.

For the $\ell_1$ -norm function $\left\|z\right\|_1$,we define its smoothing function as follow:
\begin{align*}
	\tilde{\theta}(z,\mu) =
	\begin{cases}
		|z| \quad \quad \quad if \ |z| > \mu, \\
		\frac{z^2}{2\mu} + \frac{\mu}{2} \quad if \ |z| \leq \mu,
	\end{cases}
\end{align*}

To demonstrate the performance of SAPGM, we selected the DNNM algorithm as a comparison algorithm and chose three types of problems as our benchmark tests: small-scale bi-objective optimization problems, large-scale bi-objective optimization problems with sparse structures, and tri-objective optimization problems. The objective functions in the test problem are selected from [\cite{huband2006review},\cite{lukvsan2000test},\cite{wu2023smoothing}].Now we list them in Table \ref{test problem}:

\begin{table}[H]
	\centering
	\caption{Test Problems }
	\begin{tabular}{lll}
		\toprule
		Problem & Functions & $\mathbf{x}$ \\
		\midrule
		Large scale problem &
		$\begin{aligned}
			f_1(\mathbf{x}) &= \left\| \max \{ A\mathbf{x},0\} - b \right\|_1 + 0.01 \left\|\mathbf{x}\right\|_1 \\
			f_2(\mathbf{x}) &= -\max \{ \left\|A\mathbf{x} - b\right\|_1 - \hat{\epsilon},0\} - 0.03 \left\|\mathbf{x}\right\|_1
		\end{aligned}$ 
		& $\mathbf{0} \leq \mathbf{x} \leq \mathbf{1}$ \\
		 
		 \\
		 
		CR \& MF2 & 
		$\begin{aligned}
			f_1(\mathbf{x}) &= \max \{ x_1^2 + (x_2 - 1)^2 + x_2 - 1, -x_1^2 - (x_2 - 1)^2 + x_2 + 1\} \\
			f_2(\mathbf{x}) &= -x_1 + 2(x_1^2 + x_2^2 - 1) + 1.75 |x_1^2 + x_2^2 - 1|
		\end{aligned}$ 
		& $\mathbf{1.5} \leq \mathbf{x} \leq \mathbf{2}$ \\
		
		\\
		
		CB3 \& LQ & 
		$\begin{aligned}
			f_1(\mathbf{x}) &= \max \{ x_{1}^{4} + x_{2}^{2}, (2 - x_{1})^{2} + (2 - x_{2})^{2}, 2e^{x_{2} - x_{1}} \} \\
			f_2(\mathbf{x}) &= \max \{ -x_{1} - x_{2}, -x_{1} - x_{2} + x_{1}^{2} + x_{2}^{2} - 1 \}
		\end{aligned}$ 
		& $\mathbf{1.5} \leq \mathbf{x} \leq \mathbf{2}$ \\
		
		\\
		
		CB3 \& MF1 & 
		$\begin{aligned}
			f_1(\mathbf{x}) &= \max \{ x_{1}^{4} + x_{2}^{2}, (2 - x_{1})^{2} + (2 - x_{2})^{2}, 2e^{x_{2} - x_{1}} \} \\
			f_2(\mathbf{x}) &= -x_{1} + 20 \max \{ x_{1}^{2} + x_{2}^{2} - 1, 0 \}
		\end{aligned}$ 
		& $\mathbf{0} \leq \mathbf{x} \leq \mathbf{1}$ \\
		
		\\
		
		JOS1 \& $\ell_1$ & 
		$\begin{aligned}
			f_1(\mathbf{x}) &= \frac{1}{n} \sum_{i=1}^{n} x_i^2 \\
			f_2(\mathbf{x}) &= \frac{1}{n} \sum_{i=1}^{n} (x_i - 2)^2 \\
			f_3(\mathbf{x}) &= \parallel \mathbf{x} \parallel_1
		\end{aligned}$ 
		& $\mathbf{1} \leq \mathbf{x} \leq \mathbf{2}$ \\
		
		\\
		
		BK1 \& $\ell_1$ & 
		$\begin{aligned}
			f_1(\mathbf{x}) &= x_1^2 + x_2^2 \\
			f_2(\mathbf{x}) &= (x_1 - 5)^2 + (x_2 - 5)^2 \\
			f_3(\mathbf{x}) &= \parallel \mathbf{x} \parallel_1
		\end{aligned}$ 
		& $\mathbf{-5} \leq \mathbf{x} \leq \mathbf{10}$ \\
		
		\\
		
		SP1 \& $\ell_1$ & 
		$\begin{aligned}
			f_1(\mathbf{x}) &= (x_1 - 1)^2 + (x_1 - x_2)^2 \\
			f_2(\mathbf{x}) &= (x_2 - 3)^2 + (x_1 - x_2)^2 \\
			f_3(\mathbf{x}) &= \parallel \mathbf{x} \parallel_1
		\end{aligned}$ 
		& $\mathbf{5} \leq \mathbf{x} \leq \mathbf{10}$ \\
		
		\bottomrule
		\label{test problem}
	\end{tabular}
\end{table}

For the large-scale bi-objective optimization problems with sparse structures, we selected three sparsity levels: 10\%, 20\%, and 50\%. For a given group of (m, n, Spar), the data in a large-scale problem is generated as follows:

$$
\begin{aligned}
	&\mathbf{A} = \mathbf{np.random.randn(m,n)}; &&\mathbf{s} = \mathbf{Spar * n};  \\		
	&\mathbf{x} = \mathbf{np.random.uniform(0,1,(200,1))}; && \mathbf{x[:n - int(s)] = 0}; \\	
	&\mathbf{np.random.shuffle(x)};&&\mathbf{x[x > 1] = 1}; \\	
	&\mathbf{bb} = \mathbf{A.dot(x)}; &&\mathbf{b} = \mathbf{np.maximum(bb, np.zeros(bb.shape))}.
\end{aligned}
$$

The parameter settings for the DNNM algorithm can be referenced below:
$$\sigma = 0.75,\alpha = 4,\mu_{0} = 0.5,\hat{\epsilon} = 0.001,iter_{max} = 1e3.$$
The parameter settings for the DNNM algorithm can be referenced below:
$$\varepsilon=1e-3,\delta=1e-3,c=0.25,t_0=1,iter_{max} = 1e3.$$

To demonstrate that using objective functions like JOS1 in three-objective test problems is reasonable, we compare them with the fast proximal gradient algorithm for multi-objective optimization \cite{zhang2023convergence}. For convenience, we refer to it as FPGA. This comparison shows that the SAPGM algorithm can degenerate into FPGA, thereby confirming that the composition of the three-dimensional test problems is appropriate. The results are listed in Table \ref{com purity} and Figure \ref{The Pareto fronts for Smooth problem}. It can be seen from the results that although the involvement of smoothing causes SAPGM to be slower than FPGA on smooth problems, both can obtain similar Pareto fronts. This indicates that SAPGM can degenerate into FPGA.

\begin{table}[ht]
	\centering
	\caption{Comparison between SAPGM and FPGA (Purity, Gamma, Delta, and HVS)}
	\begin{tabular}{lcccc|cccc} 
		\toprule
		Problem & \multicolumn{4}{c}{SAPGM} & \multicolumn{4}{c}{FPGA} \\
		\cmidrule(lr){2-5} \cmidrule(lr){6-9}
		& purity & $\Gamma$ & $\Delta$ & hvs
		& purity & $\Gamma$ & $\Delta$ & hvs  \\
		\midrule
		JOS1 & 0.9155 & 0.0787 & 0.8684 & 0.1163
		& 0.9155 & 0.0787 & 0.8684 & 0.1163 \\
		BK1  & \bf{0.9670} & 0.4703 & 0.9989 & \bf{0.0920}
		& 0.9520 & \bf{0.1259} & \bf{0.6884} & 0.0221 \\
		SP1  & \bf{0.9437} & \bf{0.1070} & \bf{0.6819} & \bf{0.0838}
		& 0.7370 & 0.3318 & 1.4252 & 0.0187 \\
		\bottomrule
		\label{com purity}
	\end{tabular}
\end{table}

\begin{figure}[htbp]
	\centering
	\subfloat{%
		\includegraphics[width=0.45\textwidth]{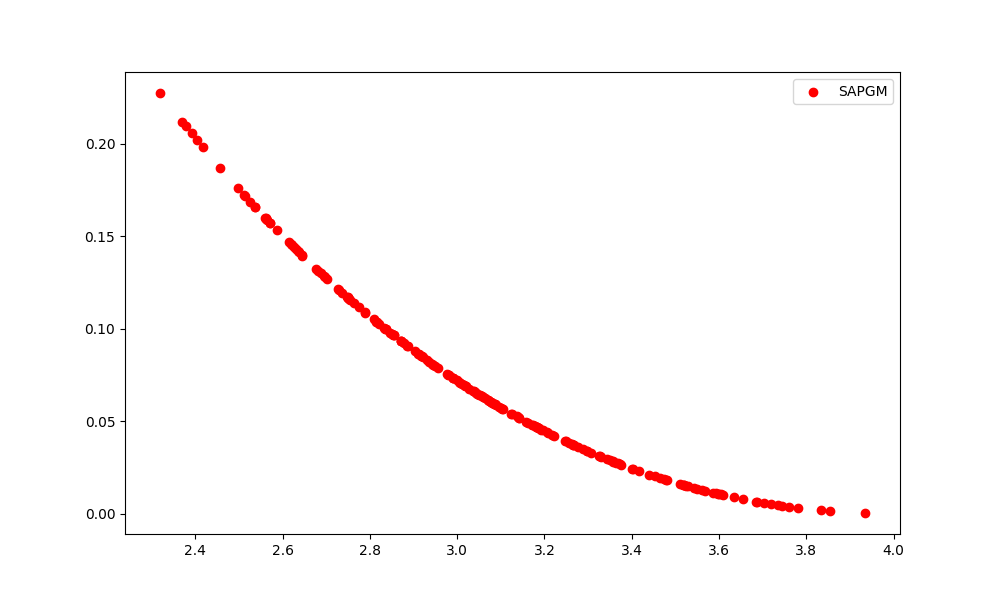}
		\label{JOS1_SAPGM}
	}
	\hfill
	\subfloat{%
		\includegraphics[width=0.45\textwidth]{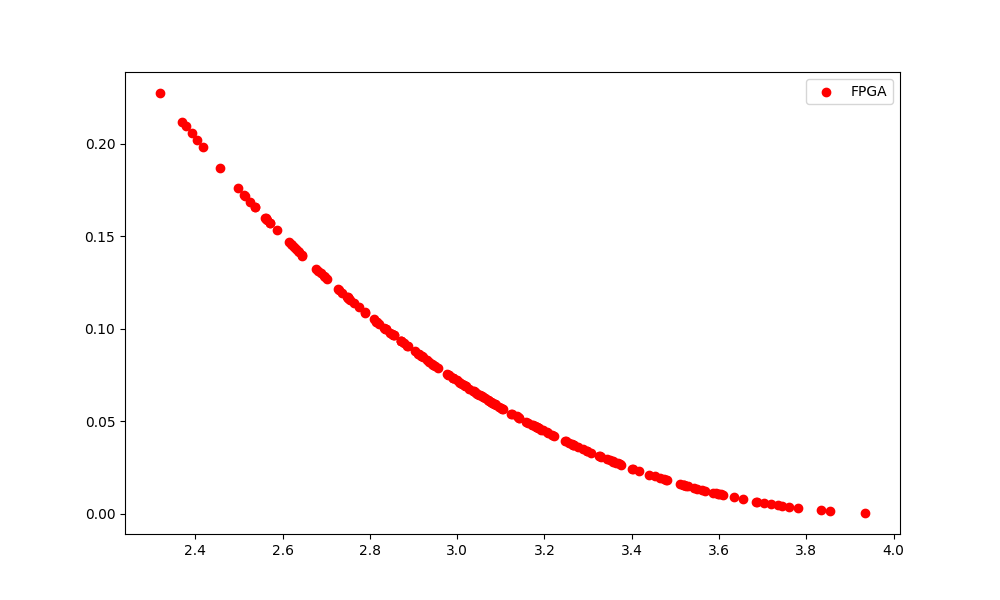}
		\label{JOS1_APG}
	}
	
\parbox{\textwidth}{\centering (a) JOS1}

	\subfloat{%
		\includegraphics[width=0.45\textwidth]{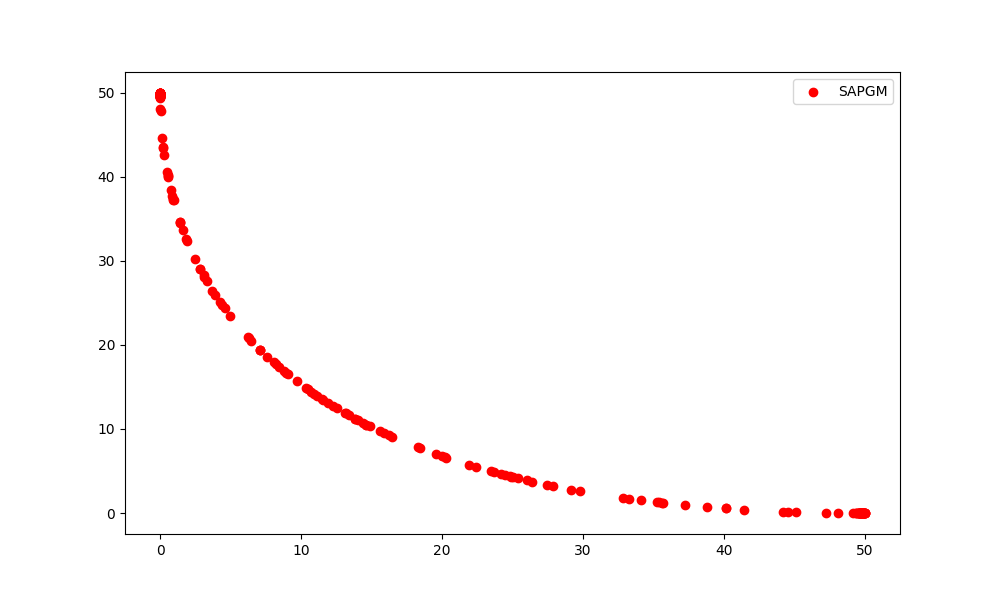}
		\label{BK1_SAPGM_2}
	}
	\hfill
	\subfloat{%
		\includegraphics[width=0.45\textwidth]{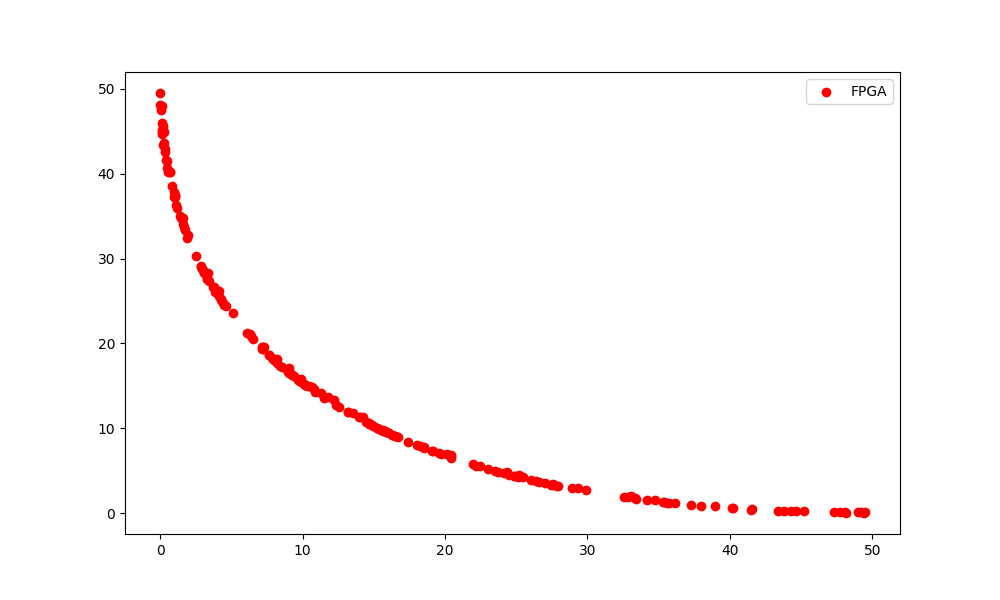}
		\label{BK1_APG}
	}

\parbox{\textwidth}{\centering (b) BK1}

	\subfloat{%
		\includegraphics[width=0.45\textwidth]{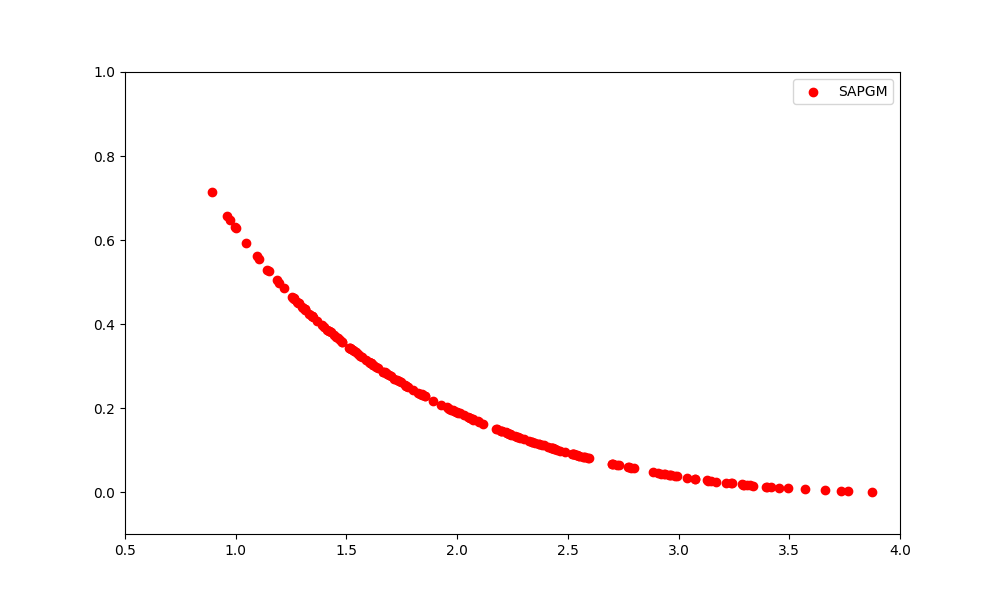}
		\label{SP1_SAPGM}
	}
	\hfill
	\subfloat{%
		\includegraphics[width=0.45\textwidth]{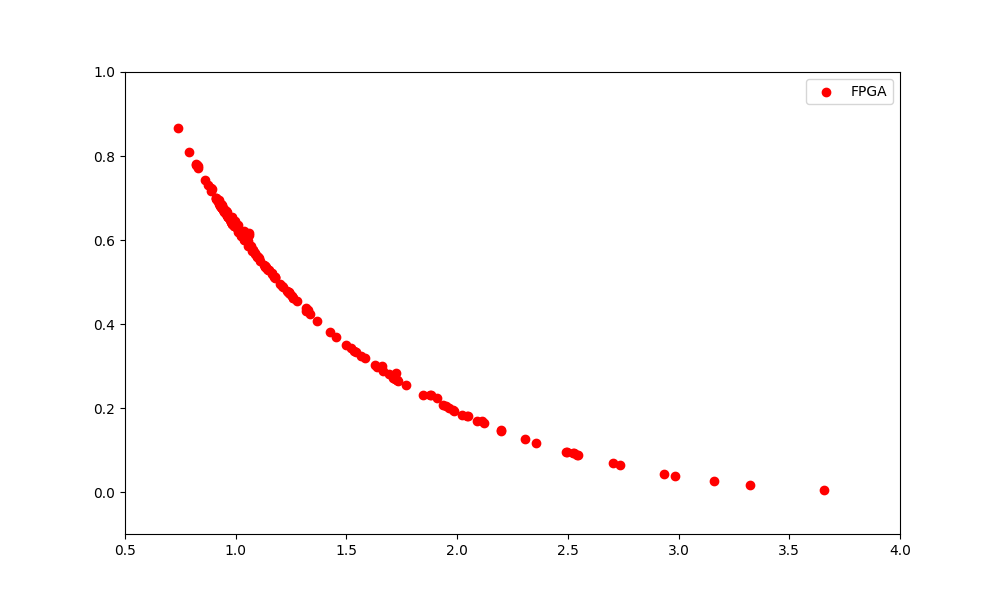}
		\label{SP1_APG}
	}
	
\parbox{\textwidth}{\centering (c) SP1}

	\caption{The Pareto fronts for Smooth problems.}
	\label{The Pareto fronts for Smooth problem}
\end{figure}

We use the following metrics to evaluate the performance of the algorithms:

\bf{Number of Iterations}: \normalfont The total number of iterations required to meet the stopping criteria.

\bf{Time}: \normalfont The time taken to satisfy the stopping criteria.

\bf{Purity} \cite{Bandtopadhyay2004multiobjective}: \normalfont This metric represents the proportion of solutions obtained by a given solver that lie within the approximated Pareto frontier.

\bf{Hypervolume}(hvs) \cite{Zitzler1999evolutionary}: \normalfont This metric quantifies the volume of the objective space dominated by the obtained Pareto frontier. 

\bf{Spread Metrics} ($\Gamma$ and $\Delta$) \cite{Custodio2011direct}: \normalfont These metrics assess the distribution of solutions across the Pareto frontier. 

Additionally, we constructed performance profiles \cite{dolan2002benchmarking} for each evaluation metric to facilitate a comprehensive comparison of the algorithms.

We now check the performance of the algorithms. For each problem above, we run the algorithms with 200 different initial points, in which Figure \ref{The Pareto fronts for large scale problems when Spar = 10} to Figure \ref{The Pareto fronts for large scale problems when Spar = 50} are the Pareto front of the large-scale double objective optimization problems, Figure \ref{The Pareto fronts for Tri-objective optimization problems} is the front of the three-objective optimization problems, and Figure \ref{The Pareto fronts for small scale two objective optimization problems} is the front of the small-scale double objective optimization problems. In general, SAPGM can map the problem ground surface well, while the DNNM algorithm can not accurately reflect the problem ground surface in some problems. Table \ref{t1} shows the average of the computational time and iteration counts for each problem. From the table, it is possible to see that acceleration is in general more efficient in terms of time. In fact, by checking the performance profiles given in Figure \ref{Performance Profile}(a) and Figure \ref{Performance Profile}(b), we observe that SAPGM performs better in terms of iteration counts and time.

Besides the performance, it is usually important to see how good the Pareto frontier is. Thus, once again we
show performance profiles, spread metric $\Gamma$ (Figure \ref{Performance Profile}(c)), spread metric $\Delta$ (Figure \ref{Performance Profile}(d)) hypervolume (Figure \ref{Performance Profile}(e)) and this time for purity (Figure \ref{Performance Profile}(f)). SAPGM outperforms the DNNM, obtaining better Pareto frontiers. We can thus conclude that at least among the test problems considered, SAPGM seems promising both in terms of performance and uniform Pareto frontiers.

In cases where the SAPGM algorithm exhibits the same number of iterations across different problems as shown in the table, we discovered that the number of iterations is related to the parameter constraints of $\mu$. As the constraints are reduced, the number of iterations changes, but this does not significantly affect the characterization of the Pareto front. Three kinds of problems, the CR\&MF2, JOS1\&$\ell_1$ and large scale problem ((m,n)=(500,100),spar=10\%), are selected as test problems to explore the influence of different $\mu$ on algorithm iteration times and Pareto frontier characterization. The results are listed in Table \ref{t11} to Table \ref{t13}.In Table \ref{t13}, hypervolume is zero due to the low sparsity of the initial point. It does not significantly affect the results. We find that with the decrease of $\mu$, the number of iterations and running time increase. However, judging from the performance profiles used before, the decrease of $\mu$ does not strengthen the characterization of the Pareto frontier but achieves slightly worse results in some problems. So we confirm that $\mu$'s choice of 1e-3 is a reasonable choice.

\begin{table}[htbp]
	\centering
	\caption{Performance of SAPGM and DNNM}
	\begin{tabular}{llrrrr}
		\toprule
		\multirow{2}{*}{Class} & \multirow{2}{*}{Problem} & \multicolumn{2}{c}{SAPGM} & \multicolumn{2}{c}{DNNM} \\
		\cmidrule(r){3-4} \cmidrule(l){5-6}
		&  & iter & time & iter & time \\
		\midrule
		\multirow{1}{*}{Two obj}
		& CR\&MF2 & \textbf{43600} & \textbf{97.6322} & 200000 & 244.7978 \\
		& CB3\&LQ & \textbf{43600} & \textbf{128.4447} & 223294 & 471.4855 \\
		& CB3\&MF1 & \textbf{43600} & \textbf{96.1398} & 150972 & 348.6537 \\
		\midrule
		\multirow{9}{*}{Large scale}
		& \textbf{Spar = 0.1} & & & & \\
		& 500*100 & \textbf{2200} & \textbf{27.4354} & 95498 & 10941.3457 \\
		& 1000*200 & \textbf{2200} & \textbf{50.6058} & 31635 & 2558.4360 \\
		& 2000*400 & \textbf{2200} & \textbf{208.9523} & 199316 & 82865.6740 \\
		& \textbf{Spar = 0.2} & & & & \\
		& 500*100 & \textbf{2200} & \textbf{32.7338} & 165255 & 11572.2238 \\
		& 1000*200 & \textbf{2200} & \textbf{65.8336} & 173010 & 21074.3461 \\
		& 2000*400 & \textbf{2200} & \textbf{322.2478} & 198621 & 41724.1614 \\
		& \textbf{Spar = 0.5} & & & & \\
		& 500*100 & \textbf{2200} & \textbf{28.2064} & 194611 & 12932.9436 \\
		& 1000*200 & \textbf{2200} & \textbf{80.7623} & 113755 & 14413.0340 \\
		& 2000*400 & \textbf{2200} & \textbf{171.9536} & 200000 & 48799.6603 \\
		\midrule
		\multirow{3}{*}{Three obj}
		& JOS1\&$\ell_1$ & \textbf{43600} & \textbf{129.6857} & 48344 & 132.5368 \\
		& BK1\&$\ell_1$ & \textbf{43600} & \textbf{346.8673} & 761652 & 1311.6791 \\
		& SP1\&$\ell_1$ & \textbf{43600} & \textbf{345.6614} & 196950 & 405.0506 \\
		\bottomrule
	\end{tabular}
	\label{t1}
\end{table}

\begin{table}[H]
	\centering
	\caption{Results for different values of $\mu$ in CR \& MF2.}
	\begin{tabular}{lcccccc}
		\toprule
		Metric & \multicolumn{6}{c}{CR \& MF2 }\\
		\cmidrule(r){2-7}
		&$\mu$=1e-1 & $\mu$=1e-2 & $\mu$=1e-3 & $\mu$=1e-5 & $\mu$=1e-7 & $\mu$=1e-10 \\
		\midrule
		Iter   & 800      & 6200     & 43600    & 200000   & 200000   & 200000 \\
		Time   & 1.1963   & 8.1796   & 93.1940  & 629.3381 & 633.3625 & 2046.4206 \\
		Purity & 0        & 0.0693   & 0.8713   & 0.8713   & 0.8713   & 0.8713 \\
		$\Gamma$ & /      & /      & 6.9795   & 6.9795   & 6.9795   & 6.9795 \\
		$\Delta$  & /      & /      & 0.8068   & 0.8068   & 0.8068   & 0.8068 \\
		hvs    & 0        & 0        & 128.0904 & 128.0904 & 128.0904 & 128.0904 \\
		\bottomrule
	\end{tabular}
	
	\label{t11}
\end{table}

\begin{table}[H]
	\centering	
	\caption{Results for different values of $\mu$ in JOS1 \& $\ell_1$.}
	\begin{tabular}{lcccccc}
		\toprule
		Metric & \multicolumn{6}{c}{JOS1 \& $\ell_1$} \\
		\cmidrule(r){2-7}
		& $\mu$=1e-1 & $\mu$=1e-2 & $\mu$=1e-3 & $\mu$=1e-5 & $\mu$=1e-7 & $\mu$=1e-10 \\
		\midrule
		Iter   & 800      & 6200     & 43600    & 200000   & 200000   & 200000 \\
		Time   & 1.9047   & 17.9726  & 129.6857 & 1644.7979 & 1645.9337 & 2146.4206 \\
		Purity & 0        & 0        & 0.9559   & 0.9559   & 0.9559   & 0.9559 \\
		$\Gamma$ & /    & /      & 0.1591   & 0.1591   & 0.1591   & 0.1591 \\
		$\Delta$ & /    & /      & 0.8635   & 0.8635   & 0.8635   & 0.8635 \\
		hvs    & 0        & 0        & 0.5866   & 0.5866   & 0.5866   & 0.5866 \\
		\bottomrule
	\end{tabular}
	
	\label{t12}
\end{table}

\begin{table}[H]
	\centering
	\caption{Results for different values of $\mu$ in the Large scale problem .}
	\begin{tabular}{lcccccc}
		\toprule
		Metric & \multicolumn{6}{c}{Large scale problem when (m,n,Spar) = (500,100,10\%)} \\
		\cmidrule(r){2-7}
		& $\mu$=1e-1 & $\mu$=1e-2 & $\mu$=1e-3 & $\mu$=1e-5 & $\mu$=1e-7 & $\mu$=1e-10 \\
		\midrule
		Iter   & 800      & 2200     & 2200    & 2200   & 2200   & 2200 \\
		Time   & 14.7241  & 50.6055  & 27.4354 & 36.1575 & 44.6184 & 40.6530 \\
		Purity & 1.0000   & 1.0000   & 0.9600  & 0.9570  & 0.9570  & 0.9570 \\
		$\Gamma$ & 0.3364 & 0.3256   & 0.3256  & 0.3256  & 0.3256  & 0.3256 \\
		$\Delta$ & 1.9225 & 2.3421   & 2.3420  & 2.3421  & 2.3421  & 2.3421 \\
		hvs    & 0.0000  & 0.0000   & 0.0000  & 0.0000  & 0.0000  & 0.0000 \\
		\bottomrule
	\end{tabular}
	
	\label{t13}
\end{table}

\begin{figure}[H]
	\centering 	
	\subfloat{%
		\includegraphics[width=0.45\textwidth]{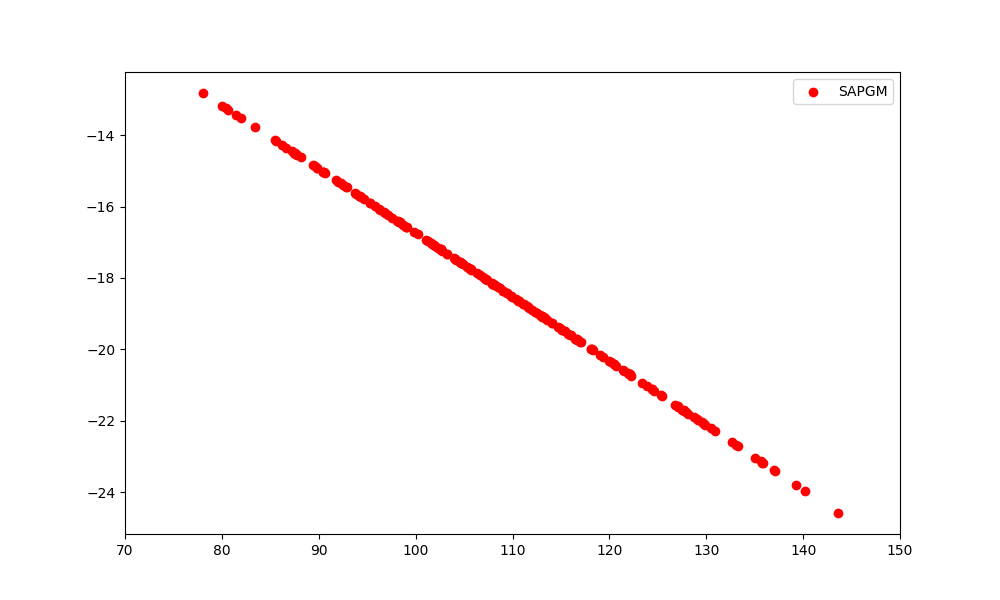}
		\label{CRSAPGM}
	}
	\hfill
	\subfloat{%
		\includegraphics[width=0.45\textwidth]{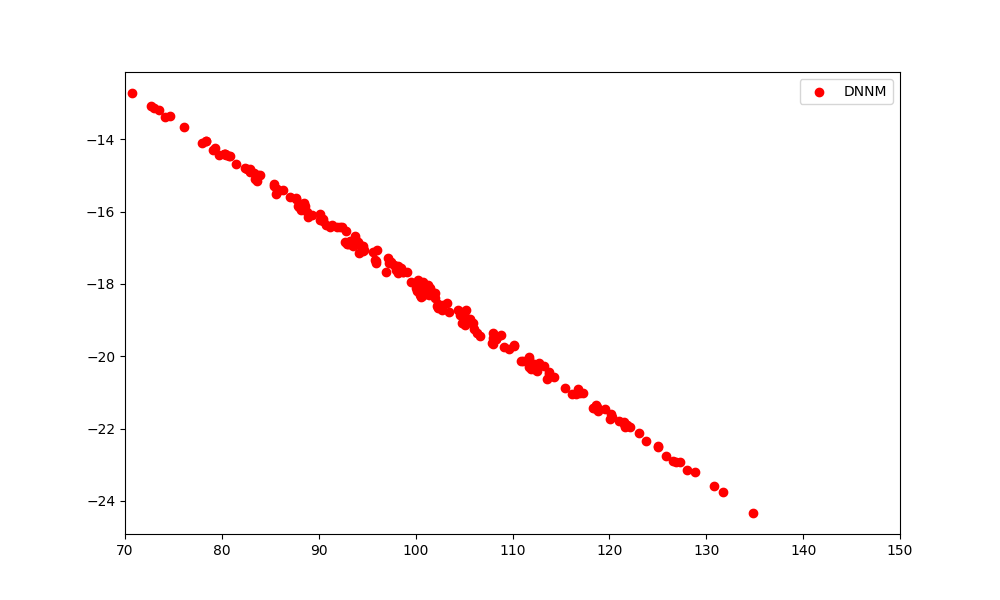}
		\label{CRDNNM}
	}

\parbox{\textwidth}{\centering (a) CR\&MF2}
	
	\subfloat{%
		\includegraphics[width=0.45\textwidth]{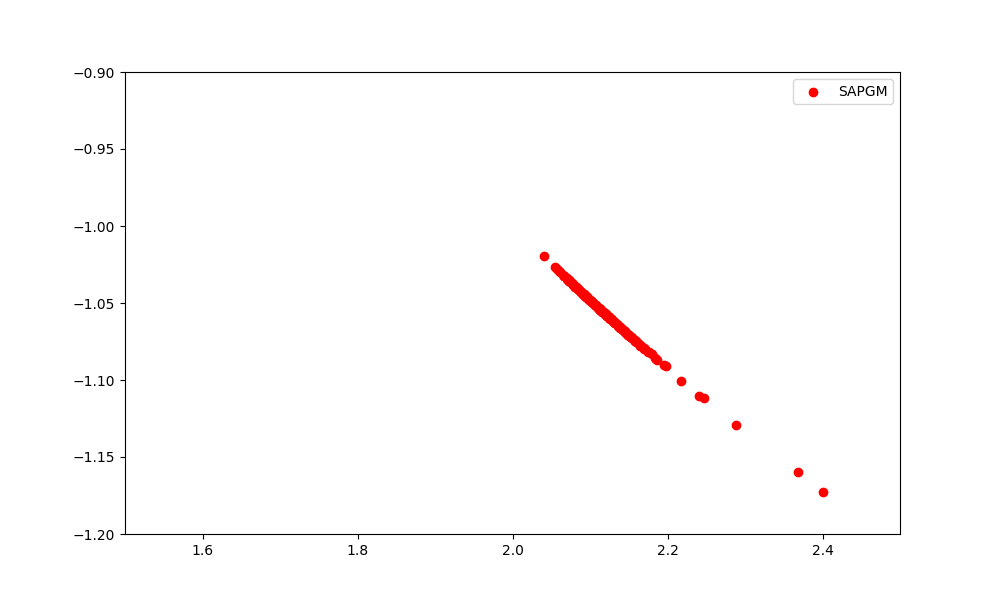}
		\label{CB3LQSAPGM}
	}
	\hfill
	\subfloat{%
		\includegraphics[width=0.45\textwidth]{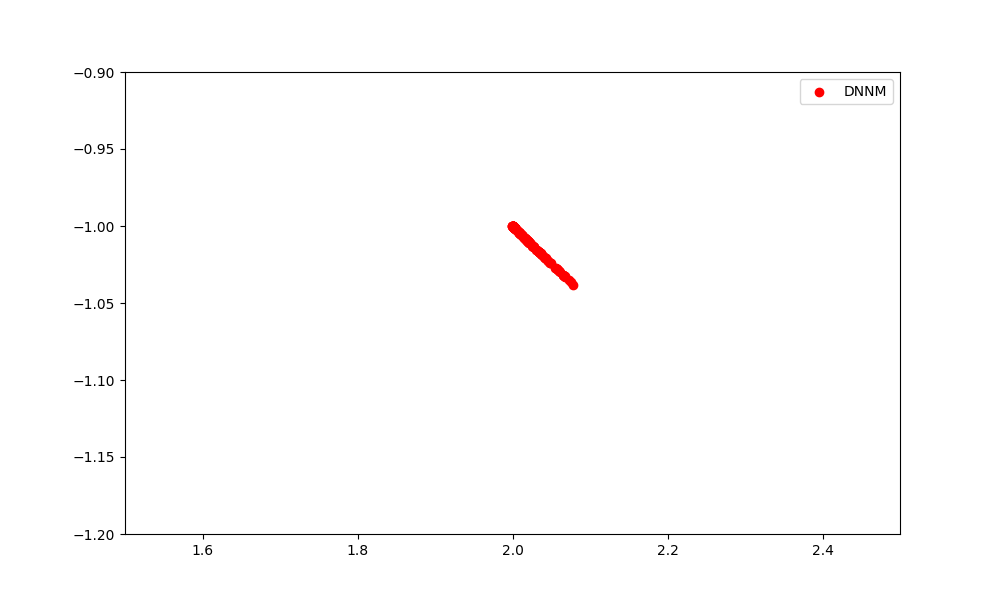}
		\label{CB3LQDNNM}
	}

	\parbox{\textwidth}{\centering (b) CB3\&LQ}
	
	\subfloat{%
		\includegraphics[width=0.45\textwidth]{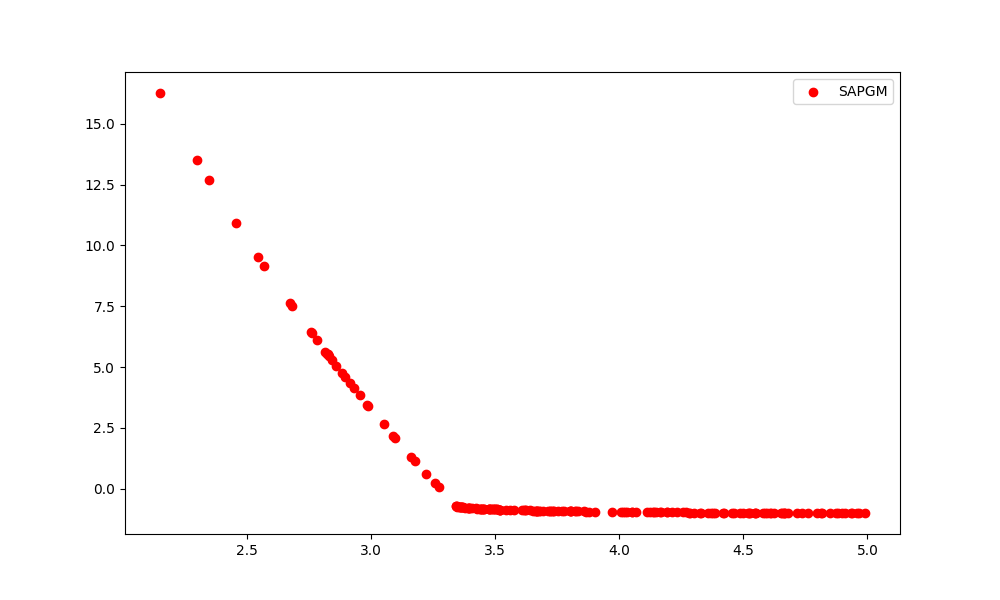}
		\label{CB3MF1SAPGM}
	}
	\hfill
	\subfloat{%
		\includegraphics[width=0.45\textwidth]{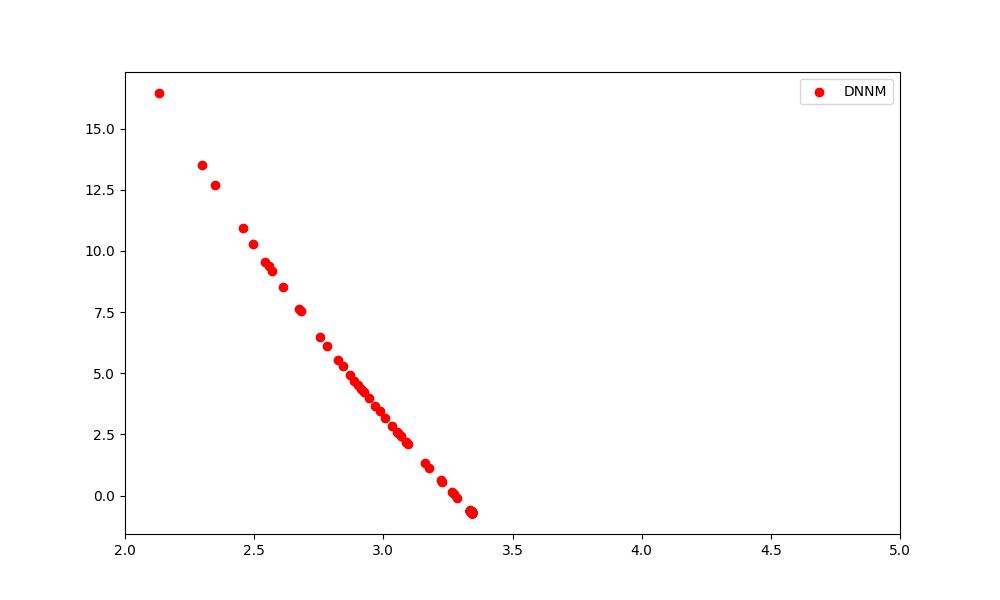}
		\label{CB3MF1DNNM}
	}

\parbox{\textwidth}{\centering (c) CB3\&MF1}
	\caption{The Pareto fronts for small scale two objective optimization problems.}
	\label{The Pareto fronts for small scale two objective optimization problems}
\end{figure}

\begin{figure}[H]
	\centering
	\subfloat{%
		\includegraphics[width=0.45\textwidth]{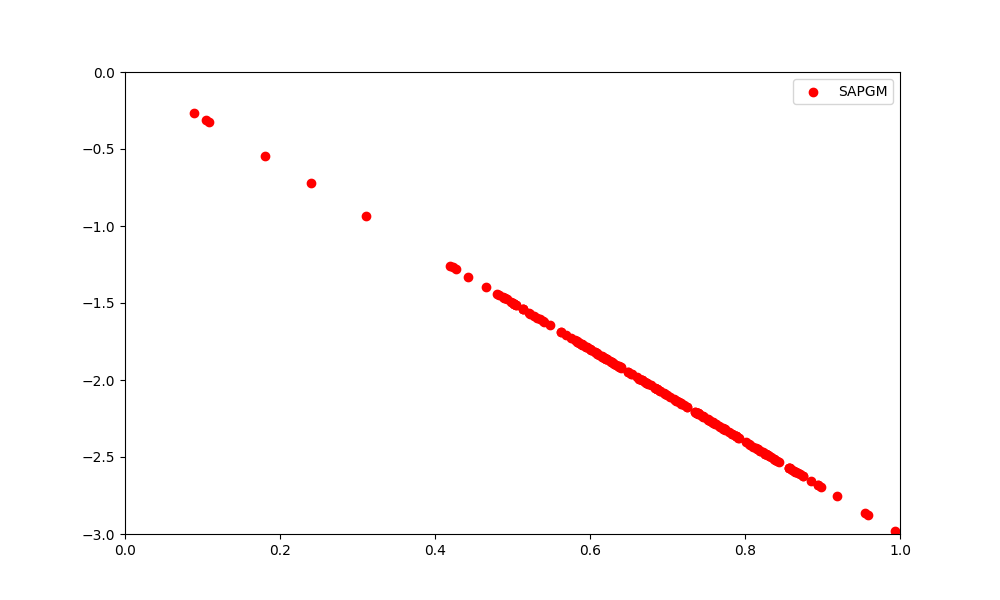}
		\label{500_100_SAPGM}
	}
	\hfill
	\subfloat{%
		\includegraphics[width=0.45\textwidth]{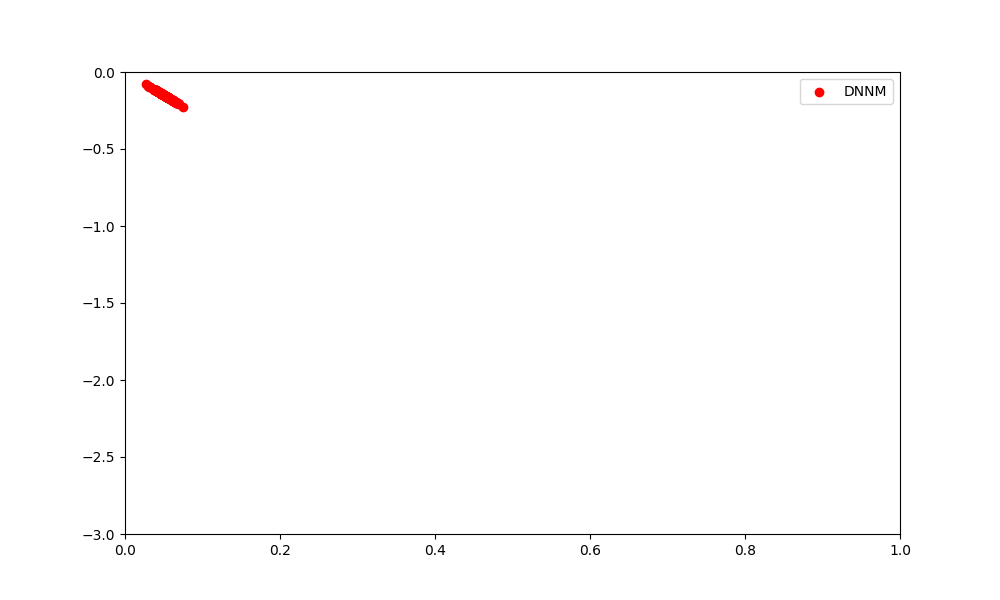}
		\label{500_100_DNNM}
	}
	
\parbox{\textwidth}{\centering (a) spar = 10\%,(m,n) = (500,100)}

	\subfloat{%
		\includegraphics[width=0.45\textwidth]{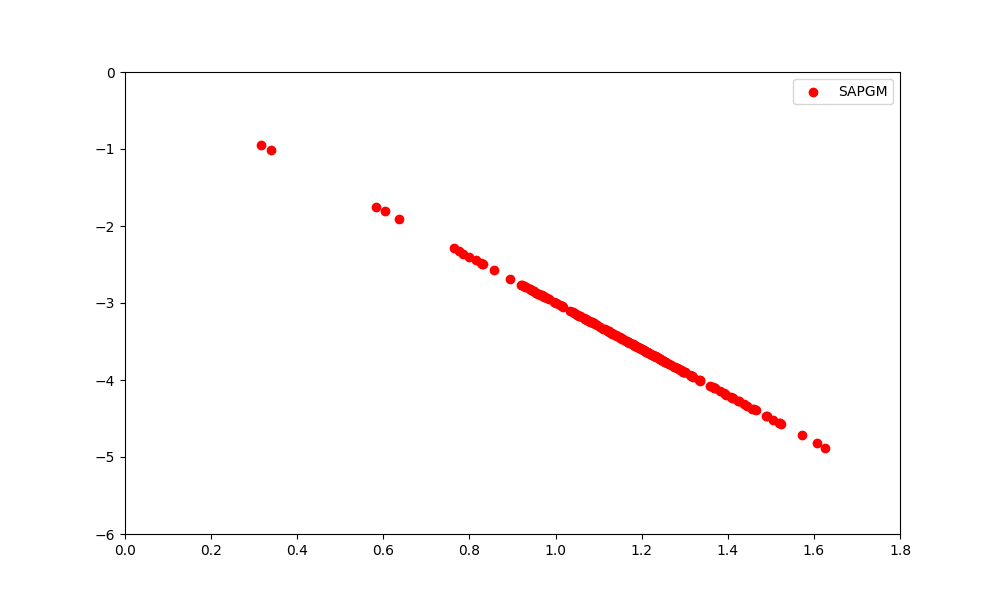}
		\label{1000_200_SAPGM}
	}
	\hfill
	\subfloat{%
		\includegraphics[width=0.45\textwidth]{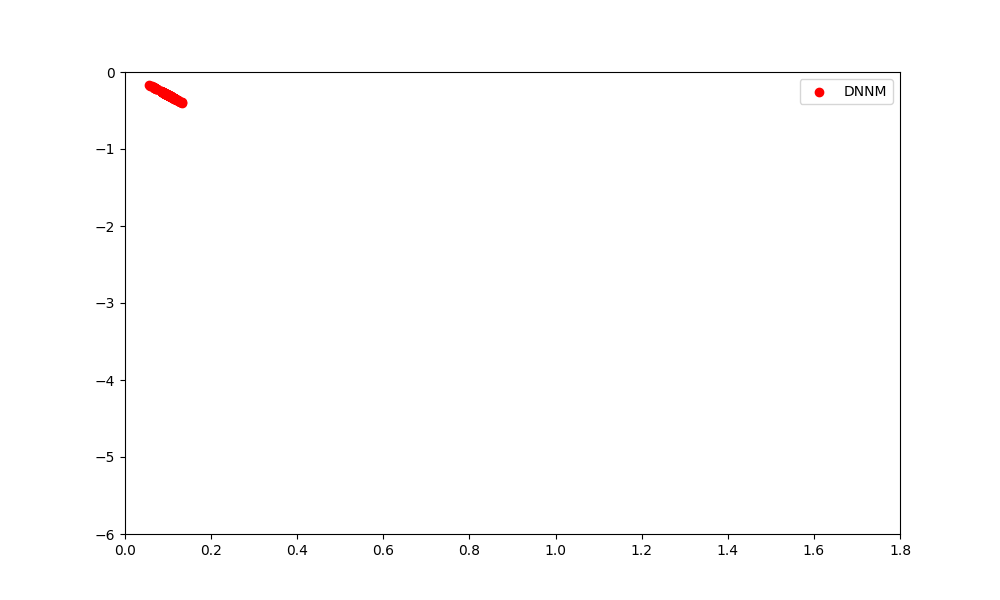}
		\label{1000_200_DNNM}
	}

\parbox{\textwidth}{\centering (b) spar = 10\%,(m,n) = (1000,200)}
	
	\subfloat{%
		\includegraphics[width=0.45\textwidth]{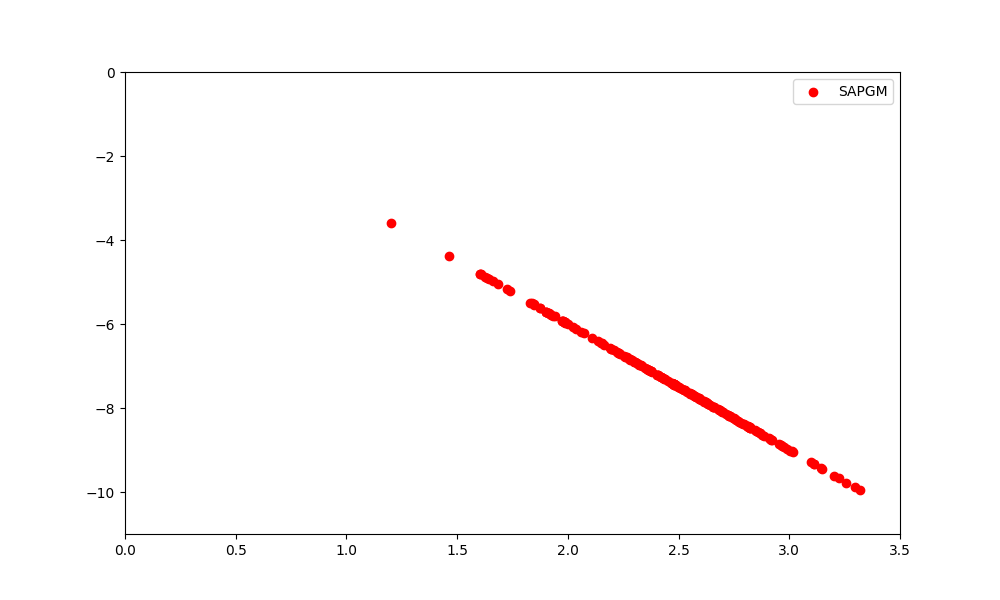}
		\label{2000_400_SAPGM}
	}
	\hfill
	\subfloat{%
		\includegraphics[width=0.45\textwidth]{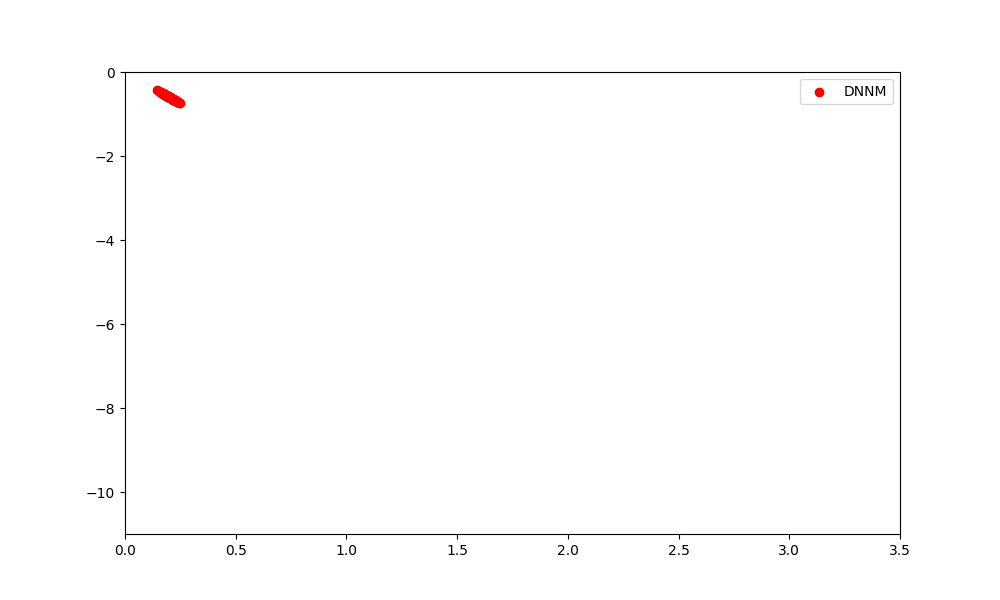}
		\label{2000_400_DNNM}
	}
	
\parbox{\textwidth}{\centering (c) spar = 10\%,(m,n) = (2000,400)}

	\caption{The Pareto fronts for large scale problems when Spar = 10\%.}
	\label{The Pareto fronts for large scale problems when Spar = 10}
\end{figure}

\begin{figure}[H]
	\centering
	\subfloat{%
		\includegraphics[width=0.45\textwidth]{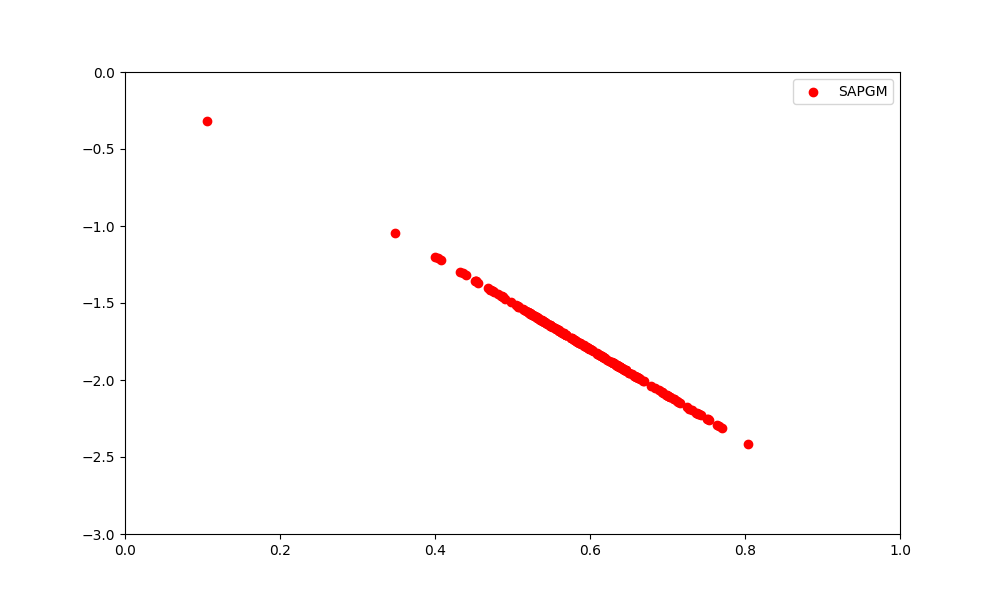}
		\label{500_100_0.2_SAPGM}
	}
	\hfill
	\subfloat{%
		\includegraphics[width=0.45\textwidth]{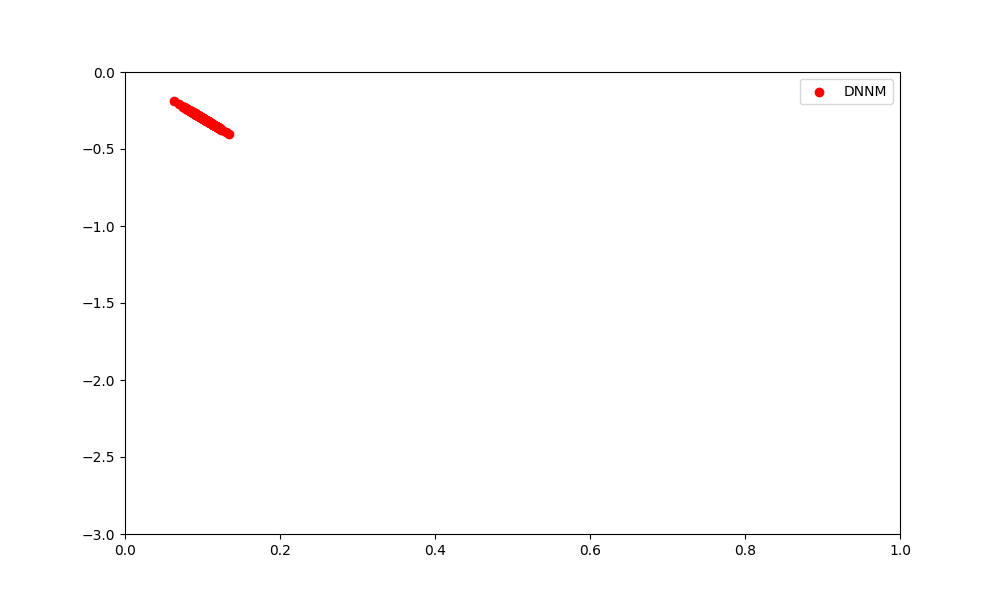}
		\label{500_100_0.2_DNNM}
	}
	
\parbox{\textwidth}{\centering (a) spar = 20\%,(m,n) = (500,100)}

	\subfloat{%
		\includegraphics[width=0.45\textwidth]{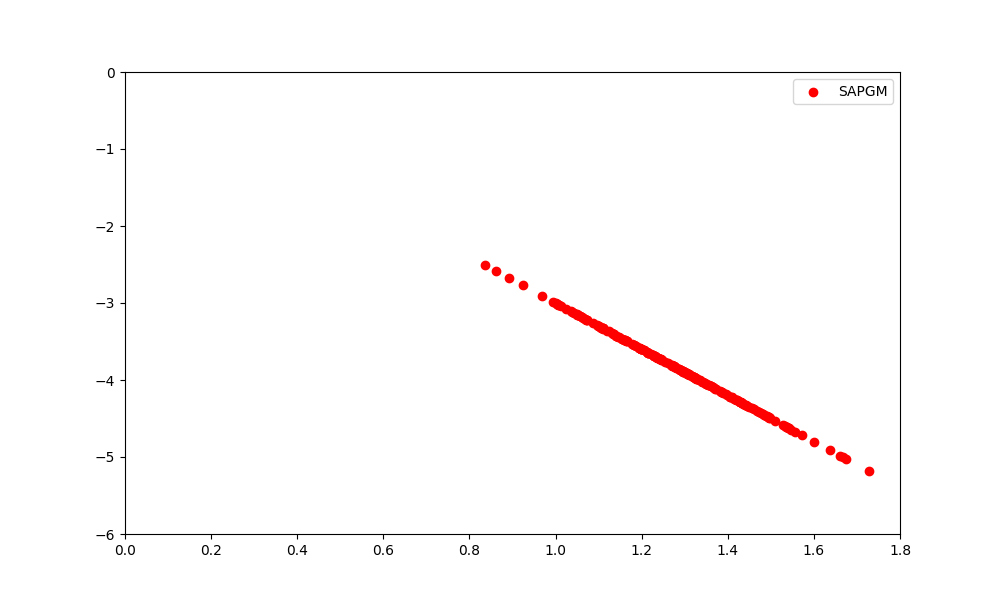}
		\label{1000_200_0.2_SAPGM}
	}
	\hfill
	\subfloat{%
		\includegraphics[width=0.45\textwidth]{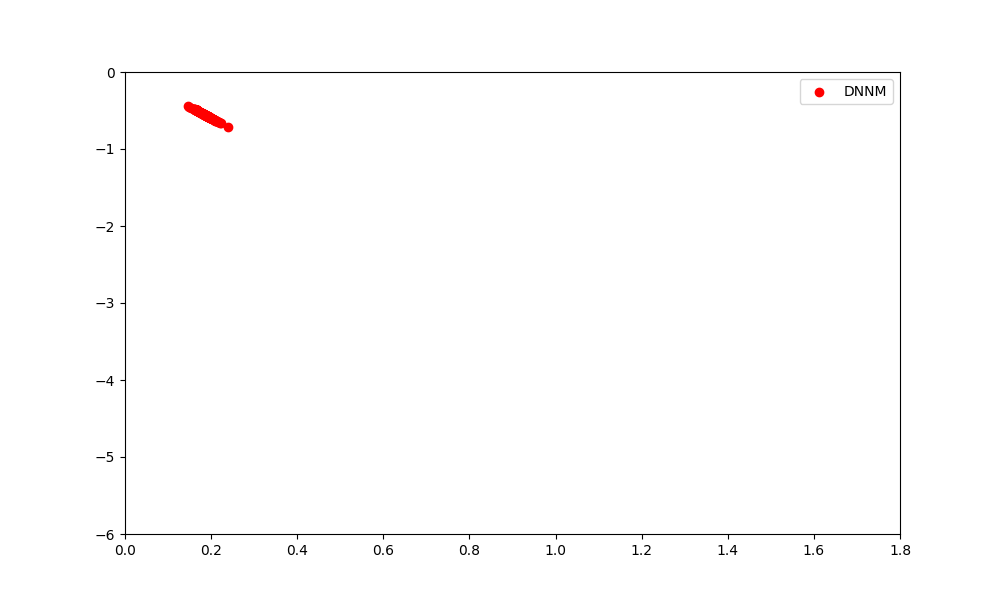}
		\label{1000_200_0.2_DNNM}
	}
	
\parbox{\textwidth}{\centering (b) spar = 20\%,(m,n) = (1000,200)}

	\subfloat{%
		\includegraphics[width=0.45\textwidth]{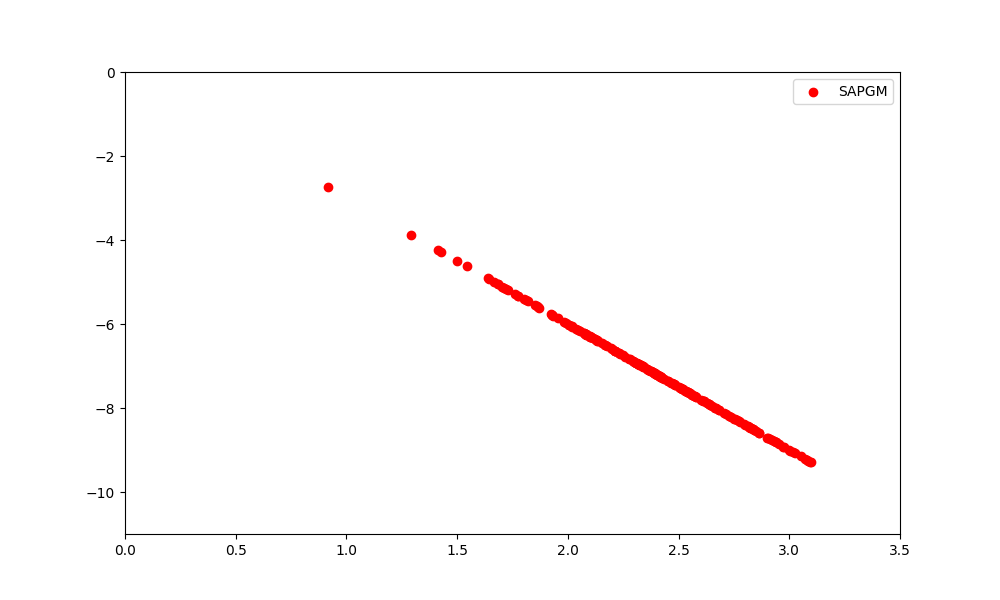}
		\label{2000_400_0.2_SAPGM}
	}
	\hfill
	\subfloat{%
		\includegraphics[width=0.45\textwidth]{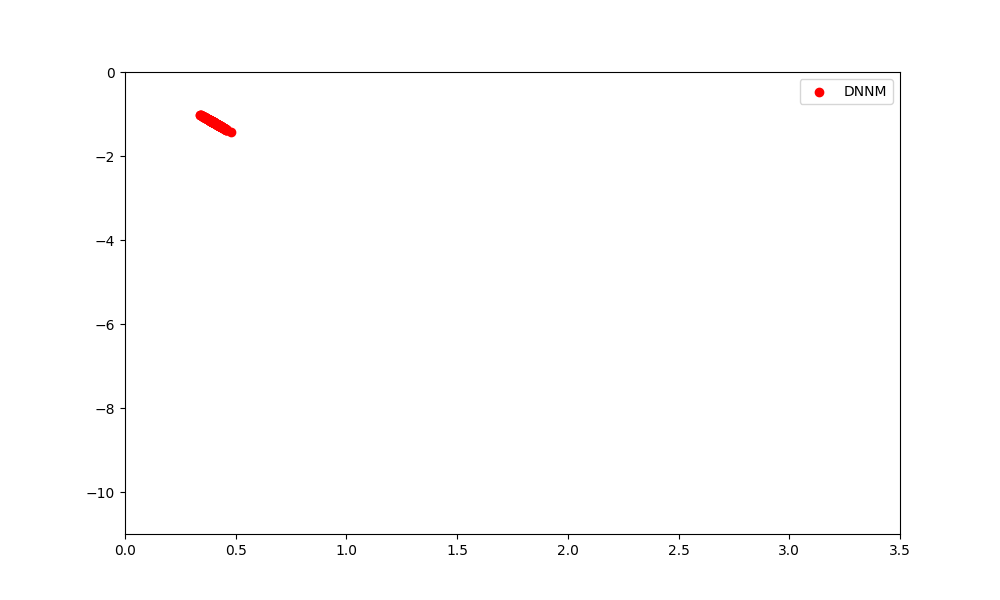}
		\label{2000_400_0.2_DNNM}
	}

\parbox{\textwidth}{\centering (c) spar = 20\%,(m,n) = (2000,400)}

	\caption{The Pareto fronts for large scale problems when Spar = 20\%.}
	\label{The Pareto fronts for large scale problems when Spar = 20}
\end{figure}

\begin{figure}[H]
	\centering
	\subfloat{%
		\includegraphics[width=0.45\textwidth]{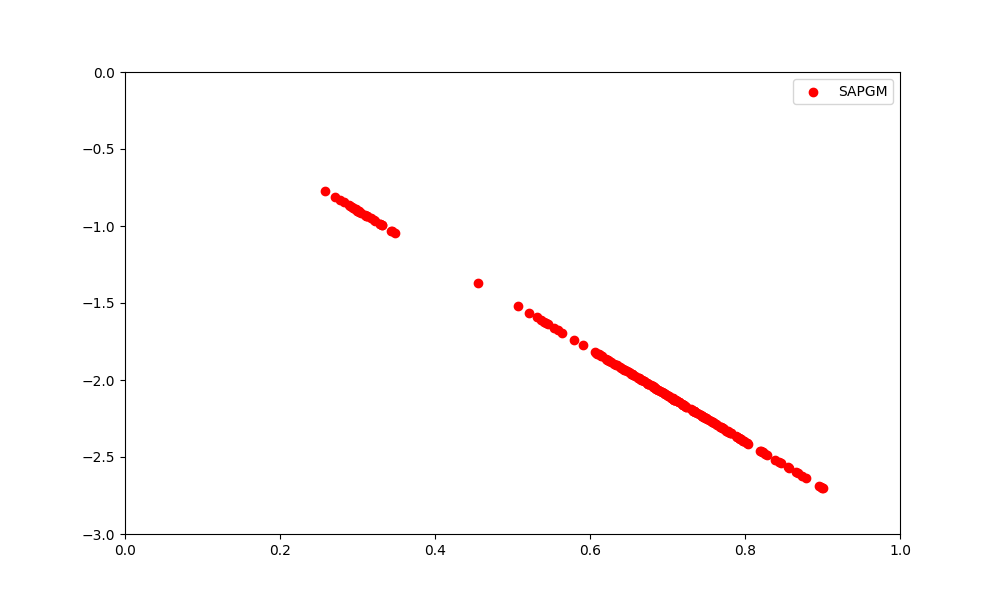}
		\label{500_100_0.5_SAPGM}
	}
	\hfill
	\subfloat{%
		\includegraphics[width=0.45\textwidth]{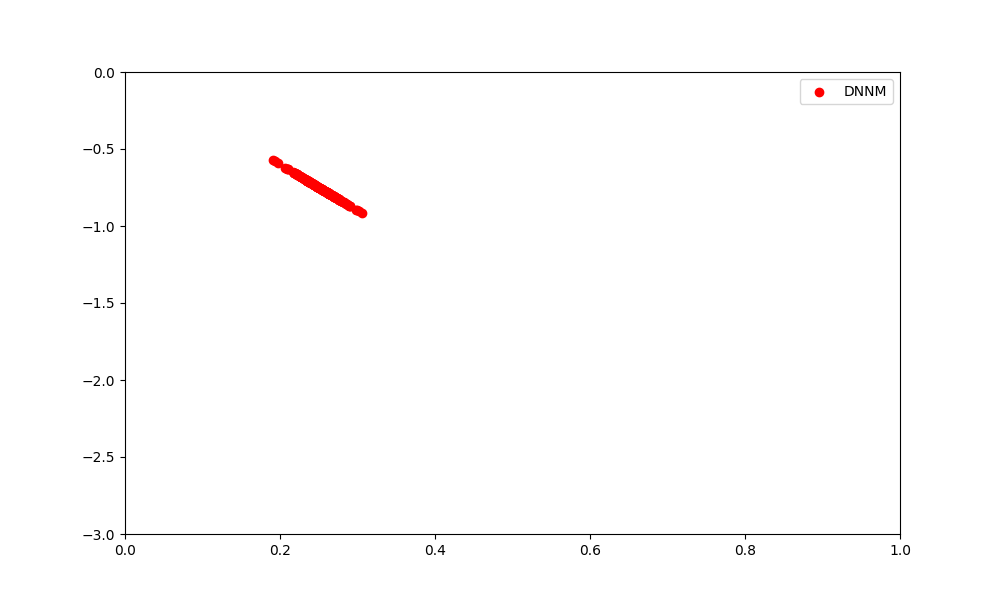}
		\label{500_100_0.5_DNNM}
	}

\parbox{\textwidth}{\centering (a) spar = 50\%,(m,n) = (500,100)}

	\subfloat{%
		\includegraphics[width=0.45\textwidth]{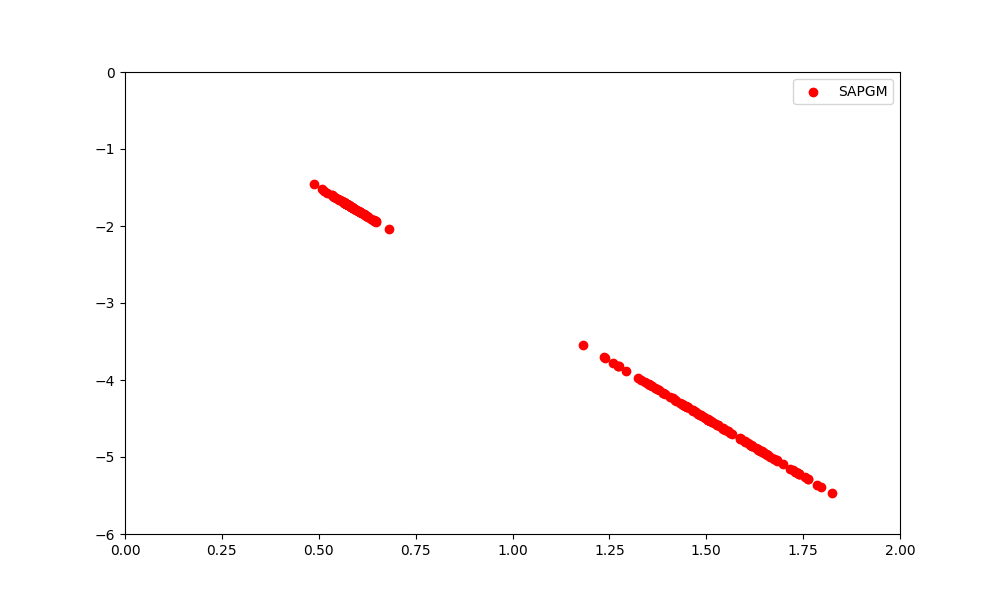}
		\label{1000_200_0.5_SAPGM}
	}
	\hfill
	\subfloat{%
		\includegraphics[width=0.45\textwidth]{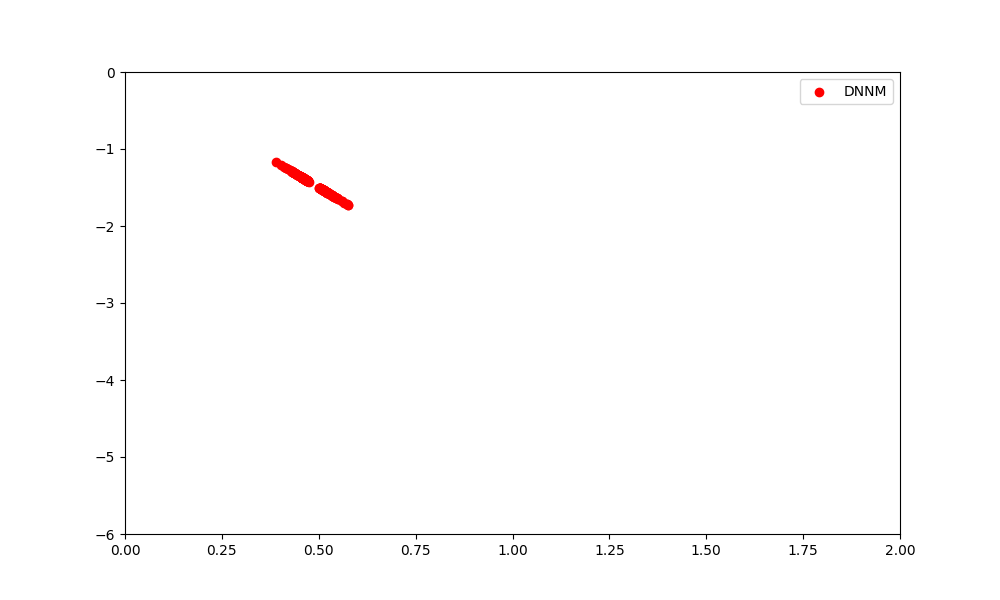}
		\label{1000_200_0.5_DNNM}
	}
	
\parbox{\textwidth}{\centering (b) spar = 50\%,(m,n) = (1000,200)}

	\subfloat{%
		\includegraphics[width=0.45\textwidth]{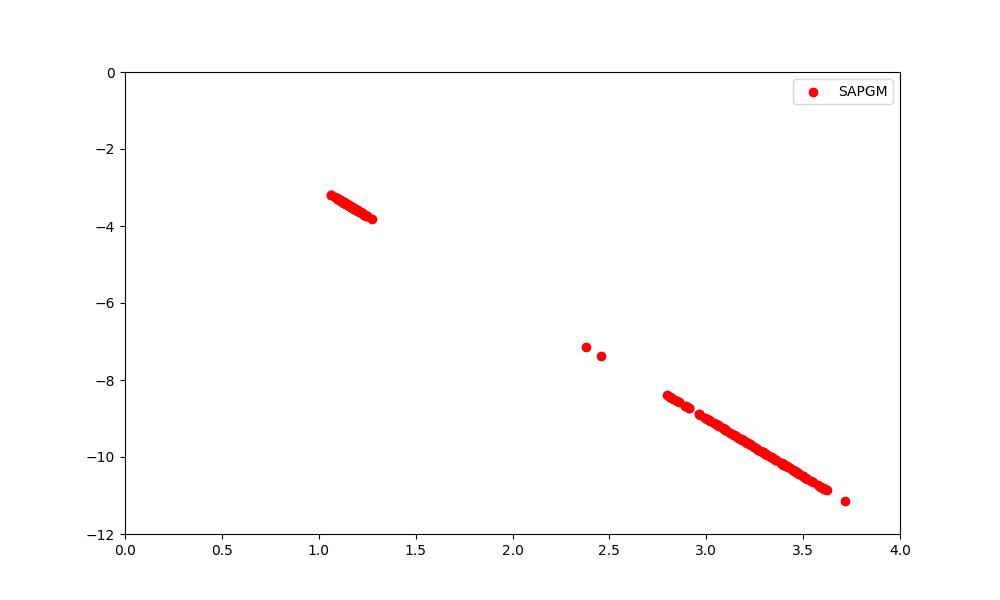}
		\label{2000_400_0.5_SAPGM}
	}
	\hfill
	\subfloat{%
		\includegraphics[width=0.45\textwidth]{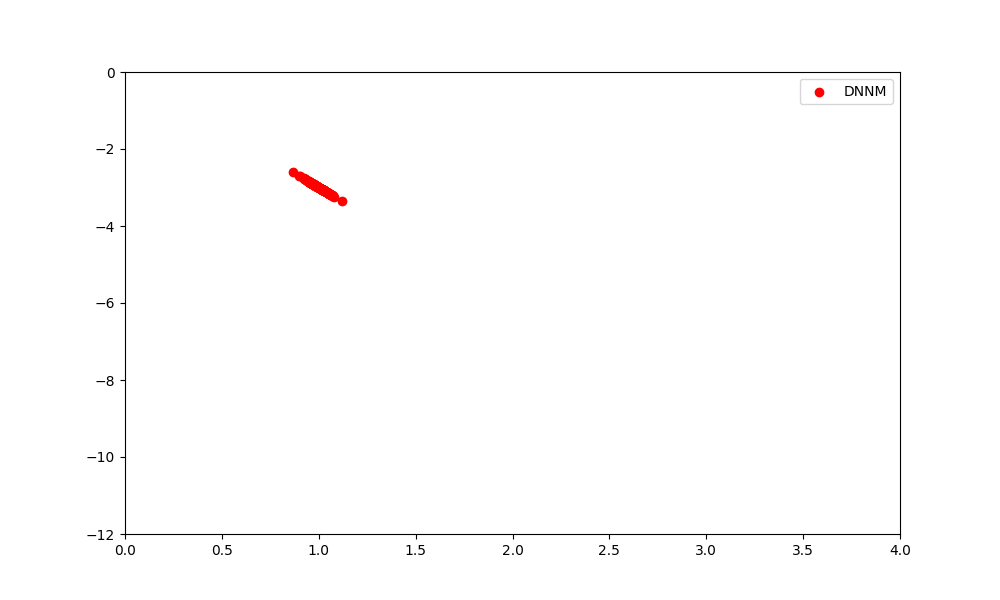}
		\label{2000_400_0.5_DNNM}
	}

\parbox{\textwidth}{\centering (c) spar = 50\%,(m,n) = (2000,400)}

	\caption{The Pareto fronts for large scale problems when Spar = 50\%.}
	\label{The Pareto fronts for large scale problems when Spar = 50}
\end{figure}

\begin{figure}[H]
	\centering
	\subfloat{%
		\includegraphics[width=0.45\textwidth]{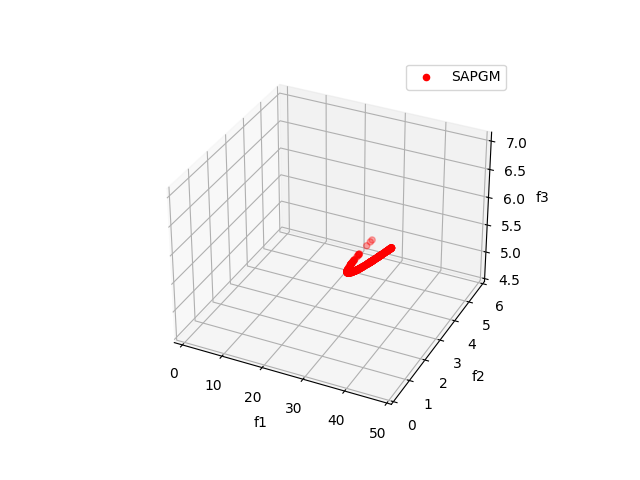}
		\label{BK1SAPGM}
	}
	\hfill
	\subfloat{%
		\includegraphics[width=0.45\textwidth]{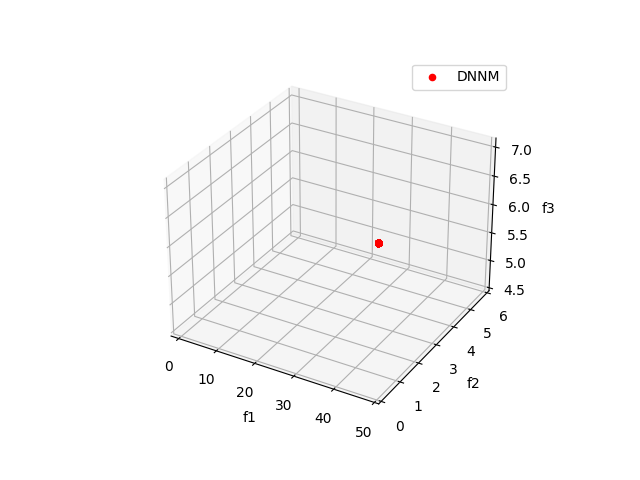}
		\label{BK1DNNM}
	}

\parbox{\textwidth}{\centering (a) BK1\&$\ell_1$}
	
	\subfloat{%
		\includegraphics[width=0.45\textwidth]{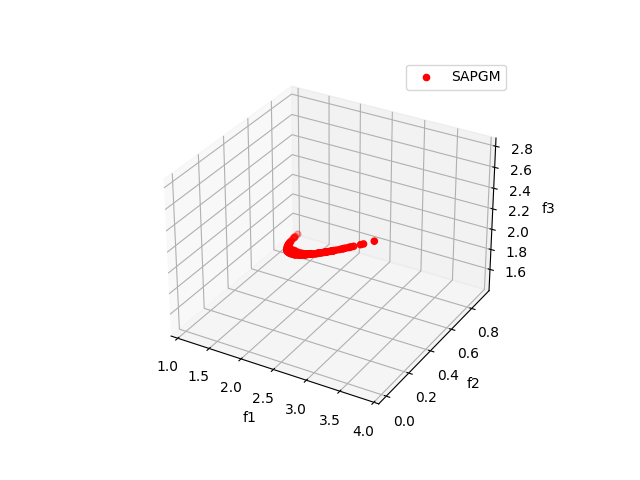}
		\label{jos1SAPGM}
	}
	\hfill
	\subfloat{%
		\includegraphics[width=0.45\textwidth]{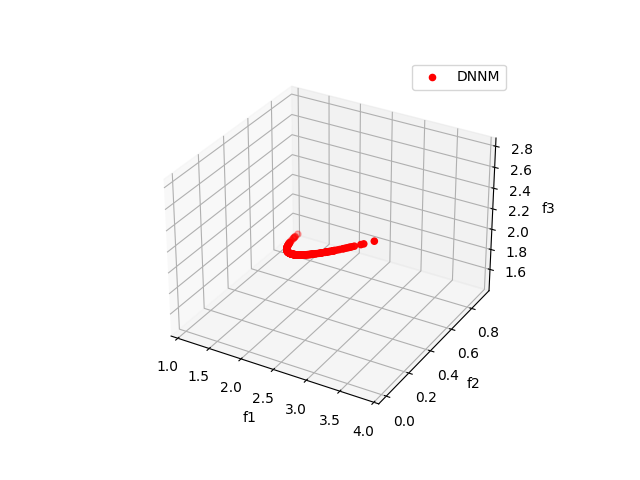}
		\label{fig:jos1DNNM}
	}
	
\parbox{\textwidth}{\centering (b) JOS1\&$\ell_1$}

	\subfloat{%
		\includegraphics[width=0.45\textwidth]{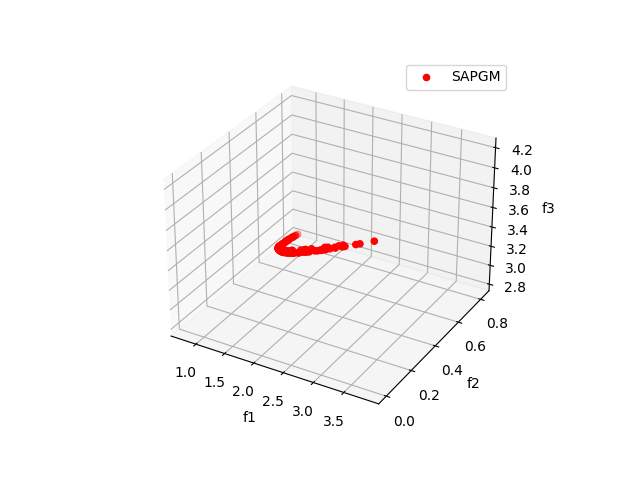}
		\label{SP1SAPGM}
	}
	\hfill
	\subfloat{%
		\includegraphics[width=0.45\textwidth]{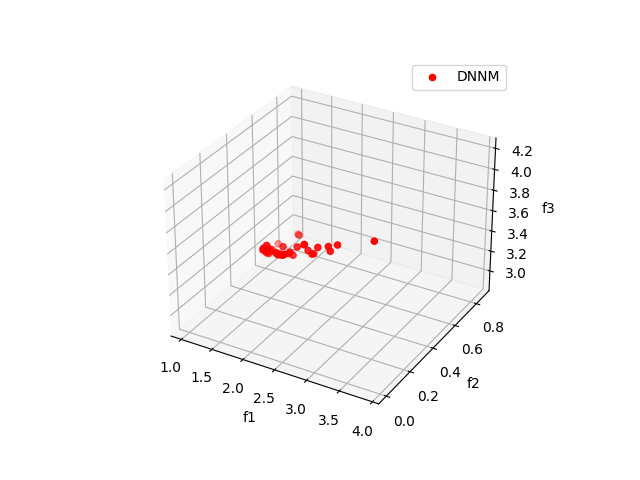}
		\label{SP1DNNM}
	}

\parbox{\textwidth}{\centering (c) SP1\&$\ell_1$}

	\caption{The Pareto fronts for Tri-objective optimization problems.}
	\label{The Pareto fronts for Tri-objective optimization problems}
\end{figure}

\begin{figure}[H]
	\centering
	\subfloat[Performance Profile:iteration]{%
		\includegraphics[width=0.45\textwidth]{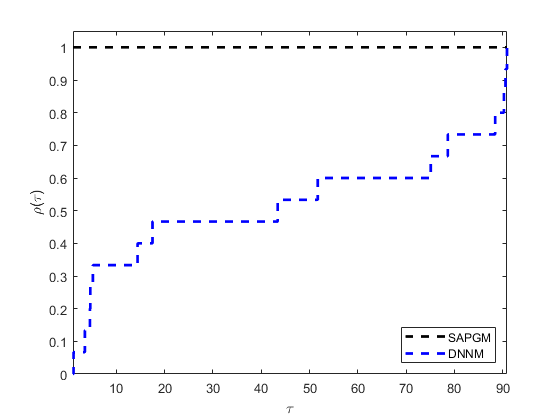}
		\label{iteration}
	}
	\hfill
	\subfloat[Performance Profile:time]{%
		\includegraphics[width=0.45\textwidth]{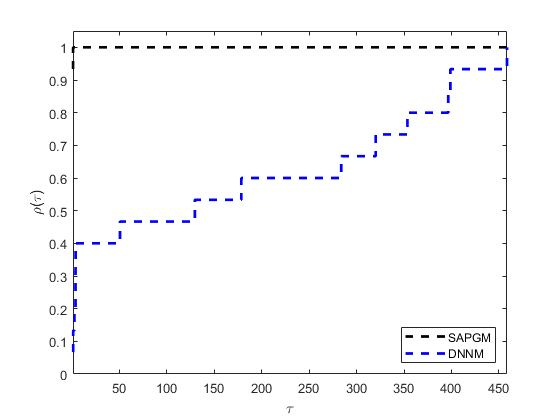}
		\label{time}
	}
	
	\subfloat[Performance Profile:Spread metric $\Gamma$]{%
		\includegraphics[width=0.45\textwidth]{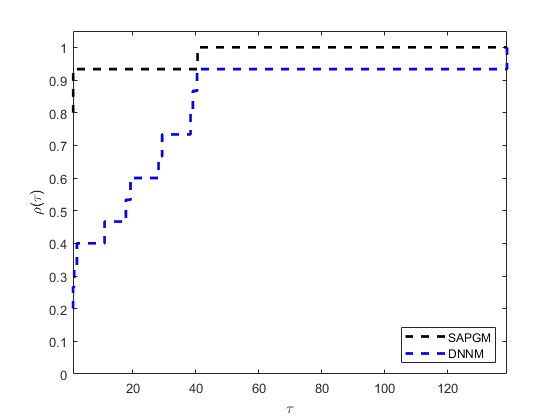}
		\label{gamma}
	}
	\hfill
	\subfloat[Performance Profile:Spread metric $\Delta$]{%
		\includegraphics[width=0.45\textwidth]{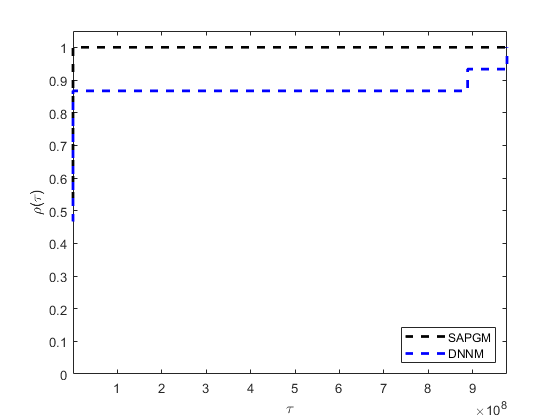}
		\label{delta}
	}
	
	\subfloat[Performance Profile:Hypervolume]{%
		\includegraphics[width=0.45\textwidth]{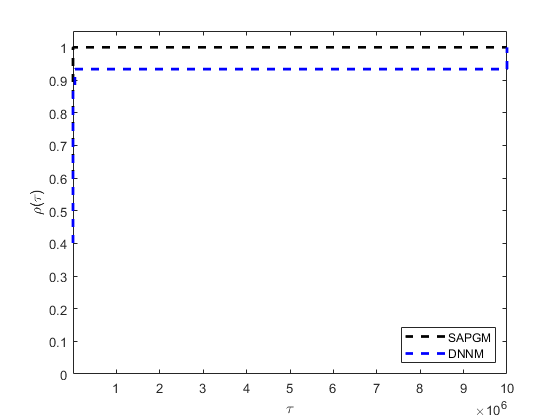}
		\label{hvs}
	}
	\hfill
	\subfloat[Performance Profile:Purity]{%
		\includegraphics[width=0.45\textwidth]{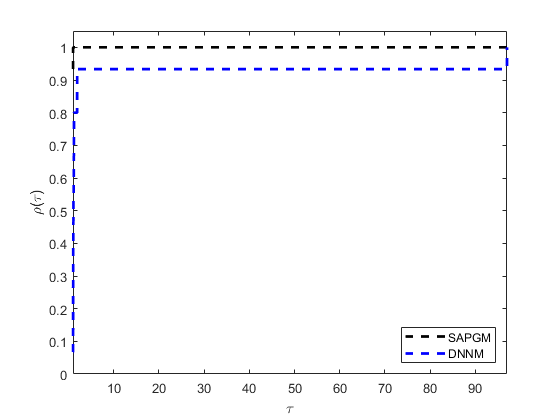}
		\label{purity}
	}
	\caption{Performance Profile.}
	\label{Performance Profile}
\end{figure}

\section{Conclusions}

In this paper, we propose a Smoothing Accelerated Proximal Gradient (SAPG) algorithm designed for the resolution of nonsmooth convex multiobjective optimization problems. Each iteration involves employing the accelerated proximal gradient with an extrapolation coefficient of $\frac{k-1}{k+\alpha-1}$ to minimize the problem (\ref{1}) with a fixed smoothing parameter, followed by an update to the smoothing parameter. Besides, we prove its convergence rate by a global energy function, which improves to $o(\ln^\sigma k/k)$. Additionally, theoretical proofs affirm that the iterates sequence converges to an optimal solution to the problem. An effective strategy for solving the subproblem is presented through its dual representation. The results of numerical experiments underscore the superior performance of the SAPG algorithm and underscore the importance of extrapolation in achieving faster convergence rates.

For future work, we plan to discuss the influence of parameters $\ell,1/\mu$, and $\alpha$ on the convergence speed of the algorithm, and give more general parameter selection criteria. This will be more conducive to the application of the algorithm to specific problems and enhance the specific application value of the algorithm. In addition, we will also study whether SAPGM has a good effect on problems with higher dimensions and more objective functions, and we hope to replace $\ell_1$ norm with $\ell_0$ norm in the target problem, which is conducive to the application of the algorithm in large-scale sparse optimization problems. But it also means that we need more theories to support our research on these goals.






\section*{Acknowledgments}
We would like to thank you for \textbf{following
the instructions above} very closely. It will save us lot of time and expedite the
process of your article's publication.









\medskip
Received xxxx 20xx; revised xxxx 20xx; early access xxxx 20xx.
\medskip

\end{document}